\documentclass[final]{siamltex}

\usepackage{times}
\usepackage{epsfig,graphics,color}
\usepackage{amssymb,amsmath} 
\usepackage[dvips, letterpaper=true, colorlinks=true, linkcolor=red, filecolor=green, citecolor=red, pdfpagemode=None]{hyperref}

\newtheorem{thm}{Theorem}

\newtheorem{lem}[thm]{Lemma}
\newtheorem{remark}[theorem]{Remark}
\newtheorem{defn}[theorem]{Definition}

\title{QUADRATIC SPLINE WAVELETS WITH SHORT SUPPORT SATISFYING HOMOGENEOUS BOUNDARY CONDITIONS \thanks{This work was supported by the SGS project "Wavelets" financed by Technical University of Liberec. The authors would like to thank Radka Szillerov\'a for her help with numerical experiments. }} 

\author{ Dana \v{C}ern\'a \thanks{Department of Mathematics and Didactics of Mathematics,
Technical University in Liberec,
Studentsk\'a 2, 461 17 Liberec, Czech Republic.
(dana.cerna@tul.cz) }
\and V\'aclav Fin\v{e}k \thanks{Department of Mathematics and Didactics of Mathematics,
Technical University in Liberec,
Studentsk\'a 2, 461 17 Liberec, Czech Republic.
(vaclav.finek@tul.cz)}
}

\begin{document}
\maketitle

\begin{abstract}
In the paper, we construct a new quadratic spline-wavelet basis on the interval and a unit square satisfying homogeneous Dirichlet boundary conditions of the first order. Wavelets have one vanishing moment and the shortest support among known quadratic spline wavelets adapted to the same type of boundary conditions. Stiffness matrices arising from a discretization of the second-order elliptic problems using the constructed wavelet basis have uniformly bounded condition numbers and the condition numbers are small. We present quantitative properties of the constructed basis. We provide numerical examples to show that the Galerkin method and the adaptive wavelet method using our wavelet basis requires smaller number of iterations than these methods with other quadratic spline wavelet bases. Moreover, due to the short support of the wavelets one iteration requires smaller number of floating point operations. 
\end{abstract}



\pagestyle{myheadings}
\thispagestyle{plain}
\markboth{D. \v{C}ERN\'A AND V. FIN\v{E}K }{QUADRATIC SPLINE WAVELETS WITH SHORT SUPPORT}

\section{Introduction}
\label{intro}

Wavelets are a powerful tool in signal analysis, image processing, and engineering applications. They are also used for numerical solution of various types of equations. 
Wavelet methods are used especially for preconditioning of systems of linear algebraic equations
arising from the discretization of elliptic problems, adaptive solving of operator equations,
solving of certain type of partial differential equations with a dimension independent convergence rate, and a sparse representation of operators. 

The quantitative properties of any wavelet method strongly depend on the used wavelet basis, namely
on its condition number, the length of the support of wavelets, the number of vanishing wavelet moments and a smoothness of basis functions. Therefore, a construction of appropriate wavelet basis is an important issue. 

In this paper, we construct a quadratic spline wavelet basis on the interval and on the unit square that is well-conditioned and adapted to homogeneous Dirichlet boundary conditions of the first order. The wavelets have one vanishing moment and the shortest possible support. Furthermore, up to our knowledge the support is the shortest among all known quadratic spline wavelets. The condition numbers of the stiffness matrices arising from the discretization of elliptic problems using the constructed basis are uniformly bounded and small. 
Let $\Omega_d=\left(0,1\right)^d$, $d=1,2$. The wavelet basis of the space $H_0^1 \left(\Omega_2 \right)$ is then obtained by an isotropic tensor product. More precisely, our aim is to propose a wavelet basis on $\Omega_d$ that satisfies the following properties:

\begin{enumerate}
\item[-] {\it Riesz basis property.} We construct Riesz bases of the space $H_0^1\left( \Omega_d \right)$.  
\item[-] {\it Locality.} The primal basis functions are local in the sense of Definition~\ref{def_wavelet_basis}.  
\item[-] {\it Vanishing moments.} The wavelets have one vanishing moment. 
\item[-] {\it Polynomial exactness.} Since the scaling basis functions are quadratic B-splines, the primal multiresolution analysis has polynomial exactness of order three. 
\item[-] {\it Short support.} The wavelets have the shortest possible support among quad\-ratic spline wavelets with one vanishing moment. 
\item[-] {\it Closed form.} The primal scaling functions and wavelets have an explicit expression. 
\item[-] {\it Homogeneous Dirichlet boundary conditions.} The wavelet basis satisfies homogeneous Dirichlet boundary conditions of the first order.
\item[-] {\it Well-conditioned bases.} The wavelet basis is well-conditioned with respect to the $H^1 \left( \Omega_d \right)$-seminorm. 
\end{enumerate} 


In \cite{Dahmen2000, Dahmen1999}, a construction of a spline-wavelet biorthogonal
wavelet basis on the interval was proposed. Both the primal and dual wavelets are local. A disadvantage of these bases was their relatively large condition number. 

Therefore many modifications of this construction were
proposed \cite{Barsch2001, Bittner2006, Burstedde2005, Talocia2000}. The construction in \cite{Primbs2010} outperforms the previous constructions for the linear and quadratic spline-wavelet bases with respect to conditioning of the wavelet bases.
In \cite{Cerna2011, Cerna2012, Dijkema2009} the construction was significantly improved also for
cubic spline wavelet basis. 

Spline wavelet bases with nonlocal duals were also constructed and adapted to various types of boundary conditions \cite{ Cerna2014a, Cerna2014b, Cerna2015, Chui1992, Jia2006, Jia2009, Jia2011, Stevenson2010}. The main advantage of these types of bases in comparison to bases with local duals are usually the shorter support of wavelets, the lower condition number of the basis and the corresponding stiffness matrices and the simplicity of the construction.

Wavelet bases of the same type as the basis in this paper are bases from \cite{Cerna2011, Dijkema2009, Primbs2010}.
The constructions from \cite{Cerna2011} and \cite{Primbs2010} lead to the same basis in the case of quadratic spline wavelet bases adapted to homogeneous Dirichlet boundary conditions of the first order. Therefore in Section~5 we compare our basis with bases from \cite{Dijkema2009, Primbs2010}.

\section{Wavelet basis on the interval}

First, we briefly review a definition of a wavelet basis, for more details about wavelet bases see \cite{Urban2009}. 
Let $H$ be a Hilbert space with the inner product $\left\langle \cdot, \cdot \right\rangle_H$ and the norm $\left\| \cdot \right\|_H$. Let $\left\langle \cdot, \cdot \right\rangle$ and $\left\| \cdot \right\|$ denote
the $L^2$-inner product and the $L^2$-norm, respectively.
Let $\mathcal{J}$ be some index set and let each index $\lambda \in \mathcal{J}$ take the form $\lambda = \left(j,k\right)$, where $\left| \lambda \right|:=j \in \mathbb{Z}$ is a {\it scale}. We define
\begin{equation*}
\left\| \mathbf{v} \right\|_2:= \sqrt{ \sum_{\lambda \in \mathcal{J}} 
 v_{\lambda}^2}, \quad {\rm for} \quad \mathbf{v}=\left\{ v_{\lambda} \right\}_{\lambda \in \mathcal{J}}, \, v_{\lambda} \in \mathbb{R},
\end{equation*}
and
\begin{equation*}
l^2 \left( \mathcal{J} \right):=\left\{ \mathbf{v}: \mathbf{v}=\left\{ v_{\lambda} \right\}_{\lambda \in \mathcal{J}}, \, v_{\lambda} \in \mathbb{R},  \left\| \mathbf{v} \right\|_2 < \infty \right\}.
\end{equation*}

Our aim is to construct a wavelet basis of $H$ in the sense of the following definition.

\begin{defn} \label{def_wavelet_basis}
A family $\Psi:=\left\{  \psi_{\lambda}, \lambda \in \mathcal{J}  \right\}$ 
is called a {\it wavelet basis} of  $H$, if
\begin{itemize}
\item[$i)$] $\Psi$ is a {\it Riesz basis} for $H$, i.e. the closure of the span of $\Psi$ is $H$
 and there exist constants $c,C \in \left( 0, \infty \right)$ such that 
\begin{equation} \label{riesz}
c \left\| \mathbf{ b} \right\|_2 \leq \left\| \sum_{\lambda \in \mathcal{J}} b_{\lambda} \psi_{\lambda} \right\|_H \leq C \left\| \mathbf{b} \right\|_2, 
\end{equation}
for all $\mathbf{b}:=\left\{ b_{\lambda} \right\}_{\lambda \in \mathcal{J}}\in l^2\left(\mathcal{J}\right)$.
\item[$ii)$] The functions are {\it local} in the sense that $\rm \mbox{diam} \ supp \ \psi_{\lambda} \leq C 2^{-\left| \lambda \right|}$ for all $\lambda \in \mathcal{J}$, and at a given level $j$ the supports of only finitely many wavelets
overlap at any point $x$. 
\end{itemize}
\end{defn}

For the two countable sets of functions $\Gamma, \Theta \subset H$ , the symbol $\left\langle \Gamma, \Theta  \right\rangle_H$ denotes the matrix 
\begin{equation*}
\left\langle \Gamma, \Theta \right\rangle_H := \left\{ \left\langle \gamma, \theta \right\rangle_H \right\}_{\gamma \in \Gamma, \theta \in \Theta }.
\end{equation*} 

\begin{remark} \label{remark1}
The constants 
\begin{equation*}
c_{\Psi}:=\text{sup} \left\{c: c \ \text{satisfies} \ (\ref{riesz}) \right\} \quad {\rm and} \quad C_{\Psi}:=\text{inf} \left\{C: C \ \text{satisfies} \ (\ref{riesz}) \right\}
\end{equation*}
are called {\it Riesz bounds} and the number $cond \ \Psi = C_{\Psi}/c_{\Psi}$ is called the {\it condition number} of~$\Psi$.
It is known that the constants $c_{\Psi}$ and $C_{\Psi}$ satisfy:
\begin{equation*}
c_{\Psi}= \sqrt{ \lambda_{min} \left( \left\langle \Psi, \Psi \right\rangle_H \right)}, \quad C_{\Psi}= \sqrt{ \lambda_{max} \left( \left\langle \Psi, \Psi \right\rangle_H \right)},
\end{equation*}
where $\lambda_{min} \left( \left\langle \Psi, \Psi \right\rangle_H \right)$ and $\lambda_{max} \left( \left\langle \Psi, \Psi \right\rangle_H \right)$ are the smallest and the largest eigenvalues of the matrix $\left\langle \Psi, \Psi \right\rangle_H$, respectively. 
\end{remark}

We define a scaling basis as a basis of quadratic B-splines in the same way as in Ref. \cite{Cerna2011, Chui1992, Primbs2010}. 
Let $\phi$ be a quadratic B-spline defined on knots $[0, 1, 2, 3]$.  It can be written explicitly as:
\begin{equation*} 
\phi(x) = \left\{ \begin{array}{cl}
\frac{x^2}{2}, & x \in [0,1], \\
-x^2+3x-\frac{3}{2}, & x \in [1,2], \\
\frac{x^2}{2}-3x+\frac{9}{2}, & x \in [2,3], \\
0, & {\rm otherwise}. \\
\end{array} \right. \end{equation*} 
The function $\phi$ satisfies a scaling equation \cite{Chui1992} 
\begin{equation} \label{primarni_skalova_rce}
\phi\left(x\right)=\frac{\phi\left(2x\right)}{4}+\frac{3 \phi\left(2x-1\right)}{4}+\frac{3\phi\left(2x-2\right)}{4}
+\frac{\phi\left(2x-3\right)}{4}.
\end{equation}
Let $\phi_b$ be a quadratic B-spline defined on knots $[0, 0, 1, 2]$, then
\begin{equation*} 
\phi_b(x) = \left\{ \begin{array}{cl}
-\frac{9 x^2}{4} + 3x, & x \in [0,1], \\
\frac{3x^2}{4}-3x+3, & x \in [1,2], \\
0, & {\rm otherwise}. \\
\end{array} \right. \end{equation*} 
The function $\phi_b$ satisfies a scaling equation \cite{Chui1992}
\begin{equation} \label{okrajova_skalova_rce}
\phi_b \left(x\right)=\frac{\phi_b\left(2x\right)}{2}+\frac{9 \phi\left(2x\right)}{8}+\frac{ 3 \phi\left(2x-1\right)}{8}.
\end{equation}
The graphs of the functions $\phi_b$ and $\phi$ are displayed in Figure~\ref{cerna_fig1}.
\begin{figure} [ht!] 
\includegraphics[height=.2\textheight]{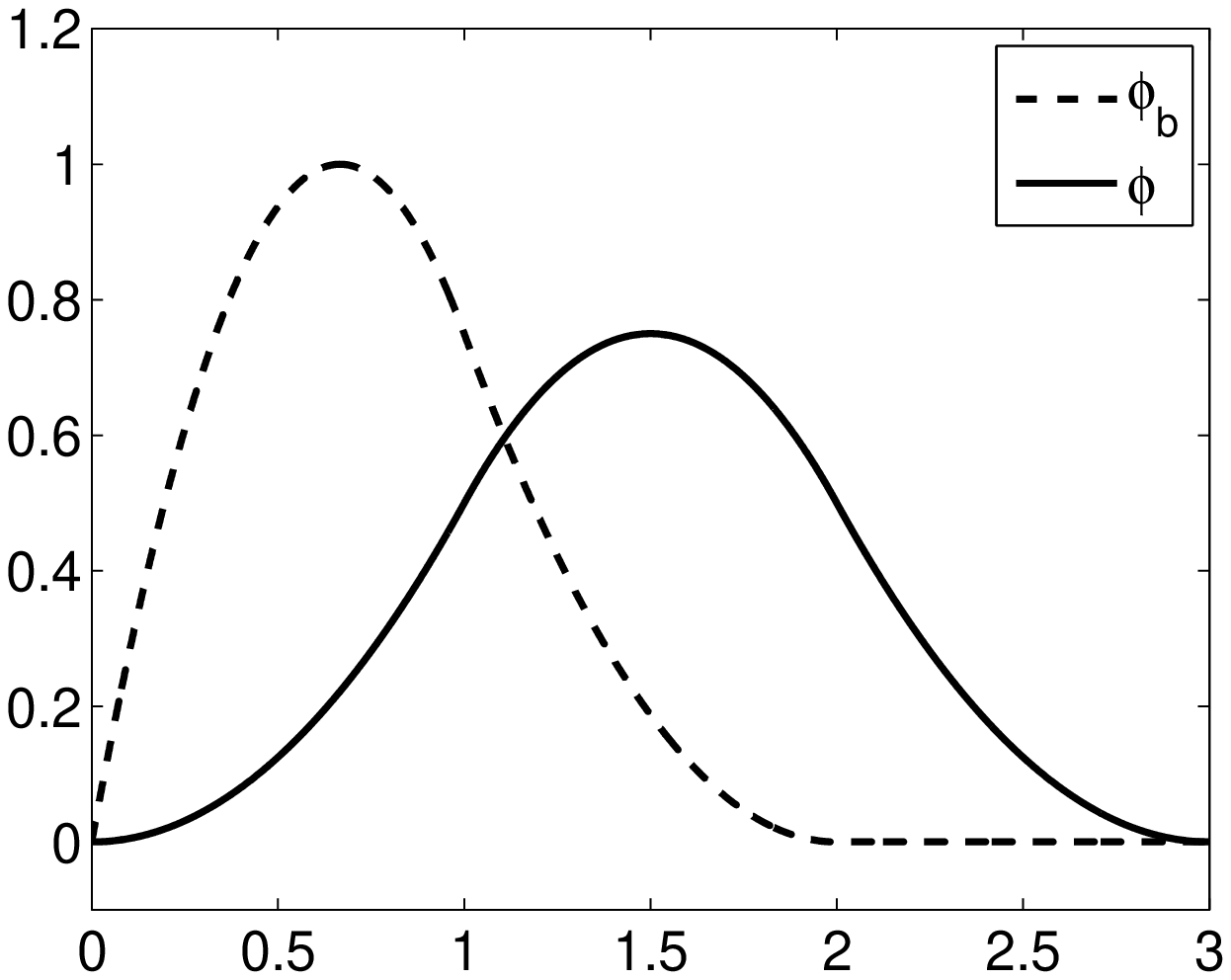}
\includegraphics[height=.2\textheight]{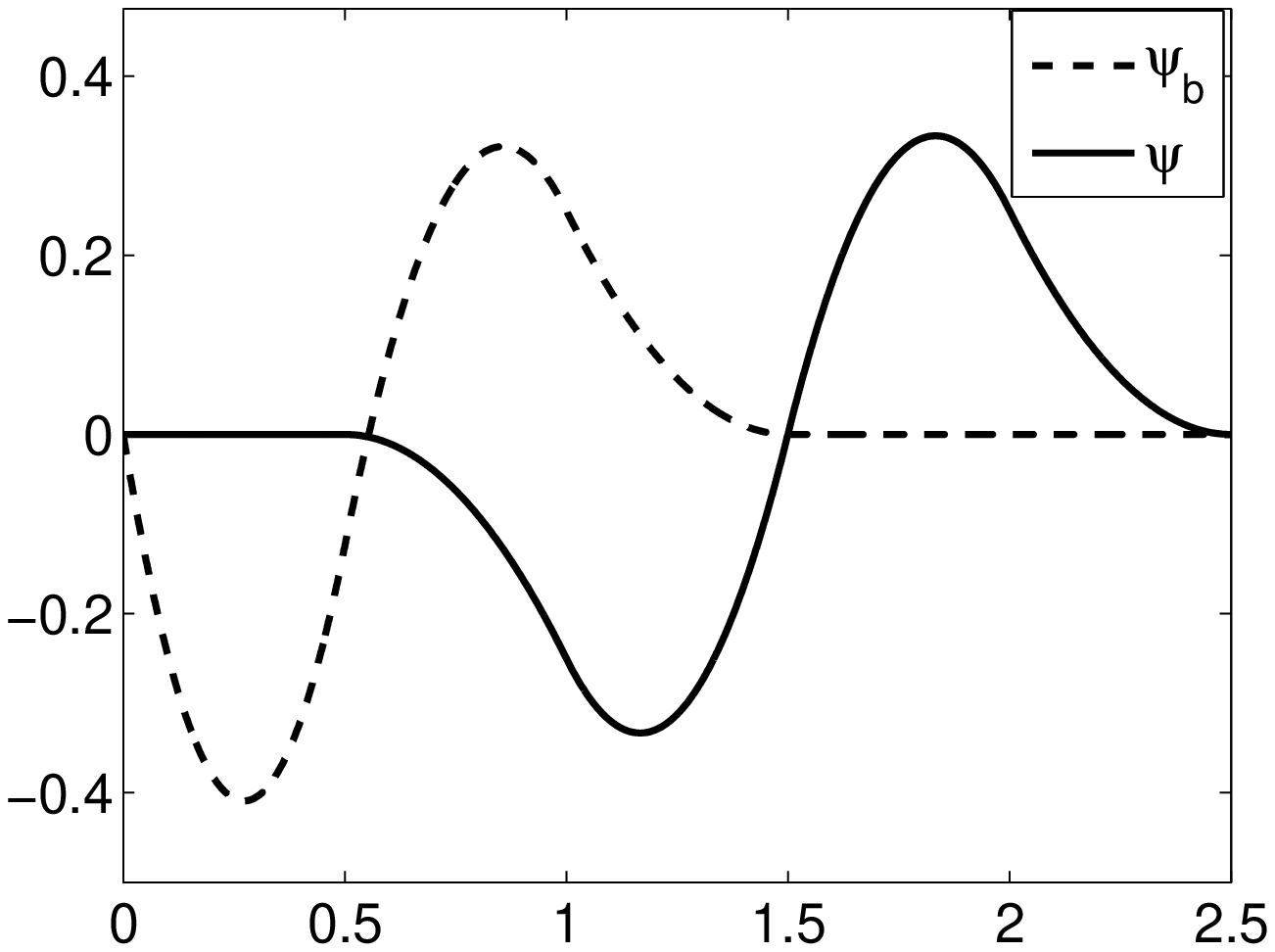}
\caption{The scaling functions $\phi$ and $\phi_b$ and the wavelets $\psi$ and $\psi_b$.}
\label{cerna_fig1}
 \end{figure}
For $j \geq 2$ and $x \in \left[0,1\right]$ we set 
\begin{eqnarray} \label{dilatacephi}
\phi_{j,k}(x) &=& 2^{j/2}\phi(2^j x-k+2),  k = 2,...,2^j-1, \\
\nonumber \phi_{j,1}(x) &=& 2^{j/2} \phi_b (2^j x),  \quad  \phi_{j,2^j}(x) =  2^{j/2} \phi_b (2^j (1-x)).
\end{eqnarray} 
We define a wavelet $\psi$ and a boundary wavelet $\psi_b$ as
\begin{equation} \label{waveletova_rce}
\psi(x) = - \frac{1}{2} \phi(2x-1)+\frac{1}{2}\phi(2x-2) \quad {\rm and} \quad \psi_b(x) = \frac{ - \phi_b(2x) }{2} + \frac{\phi(2x)}{2}.
\end{equation}
Then ${\rm supp} \, \psi = \left[0.5, 2.5 \right]$, ${\rm supp} \, \psi_b = \left[0, 1.5 \right]$, 
and both wavelets have one vanishing moment, i.e.
\begin{equation*}
\int_{-\infty}^{\infty} \psi (x) dx =0 \quad {\rm and} \quad \int_{-\infty}^{\infty} \psi_b (x) dx =0. 
\end{equation*}
The graphs of the wavelet $\psi$ and the boundary wavelet $\psi_b$ are displayed in Figure~\ref{cerna_fig1}. 
For $j \geq 2$ and $x \in \left[0,1 \right]$ we define
\begin{eqnarray} \label{dilatacepsi}
\psi_{j,k}(x) &=& 2^{j/2}\psi(2^j x-k+2),  k = 2,...,2^j-1, \\
\nonumber \psi_{j,1}(x) &=& 2^{j/2}\psi_b (2^j x),  \quad  \psi_{j,2^j}(x) =  -2^{j/2}\psi_b (2^j (1-x)).
\end{eqnarray}  
We denote the index sets by
\begin{equation*}
\mathcal{I}_j = \left\{ k \in \mathbb{Z}: 1 \leq k \leq 2^j\right\}.
\end{equation*}
We define 
\begin{equation*}
\Phi_j = \left\{\phi_{j,k}, k \in \mathcal{I}_j \right\}, \quad
\Psi_j = \left\{\psi_{j,k}, k \in \mathcal{I}_j  \right\},
\end{equation*}
and
\begin{equation} \label{definition_Psi}
\Psi=\Phi_2 \cup \bigcup_{j=2}^{\infty} \Psi_j, \quad \Psi_s=\Phi_2 \cup \bigcup_{j=2}^{1+s} \Psi_j.
\end{equation}
In Section \ref{sectionRB} we prove that $\Psi$, when normalized with respect to the $H^1$--seminorm,
forms a wavelet basis of the Sobolev space $H_0^1 \left(0,1 \right)$. 

\section{Refinement matrices} \label{section_refinement_matrices}
By (\ref{primarni_skalova_rce}), (\ref{okrajova_skalova_rce}), (\ref{dilatacephi}), (\ref{waveletova_rce}) and (\ref{dilatacepsi}), there exist
{\it refinement matrices} $\mathbf{M}_{j,0}$ and $\mathbf{M}_{j,1}$ 
such that
\begin{equation} \label{ref_matrices}
\Phi_j = \mathbf{M}_{j,0}^T \Phi_{j+1}, \quad \Psi_j = \mathbf{M}_{j,1}^T \Phi_{j+1}.
\end{equation}
In these formulas we view the sets $\Phi_j$ and $\Psi_j$ as column
vectors with entries $\phi_{j,k}$ and $\psi_{j,k}$, $k \in \mathcal{I}_j$,
respectively.

Due to (\ref{primarni_skalova_rce}) and (\ref{okrajova_skalova_rce}), the refinement matrix $\mathbf{M}_{j,0}$ has the following structure: 
\begin{equation*}
\mathbf{M}_{j,0} = \left[ \begin{array}{c|c|c}
 & \multicolumn{2}{c}{} \\
\cline{2-2} \mathbf{M}_L  &  & \vspace{-3mm} \\ 
            & & \\
\cline{1-1} & \hspace{3mm} \mathbf{M}_{j,0}^I \hspace{3mm} & \\
\cline{3-3} & &   \\ 
\cline{2-2} \multicolumn{2}{c|}{} &  \mathbf{M}_R   \vspace{-3mm} \\ 
    \multicolumn{2}{c|}{} &   
\end{array} \right].
\end{equation*}
where $\mathbf{M}_{j,0}^I $ is a $2^{j+1} \times 2^j $ matrix given by
\begin{equation*}
\left( \mathbf{M}_{j,0}^I  \right)_{m,n} = \begin{cases} \frac{ h_{m+2-2n} }{\sqrt{2}}, & n=1, \ldots, 2^j, \ 1 \leq m+2-2n \leq 4, \\
 0, & {\rm otherwise},
\end{cases}
\end{equation*} 
where 
\begin{equation*}
\mathbf{h}=\left[h_1, h_2, h_3, h_4  \right]= \left[ \frac{1}{4}, \frac{3}{4}, \frac{3}{4}, \frac{1}{4} \right]
\end{equation*}
is a vector of coefficients from the scaling equation (\ref{primarni_skalova_rce}).
The matrix $\mathbf{M}_L$ is given by
\begin{equation*}
\mathbf{M}_L = \frac{1}{\sqrt{2}}  \mathbf{h}_b^T, \quad {\rm where} \quad \mathbf{h}_b=\left[h_1^b, h_2^b, h_3^b \right]= \left[ \frac{1}{2}, \frac{9}{8}, \frac{3}{8} \right]
\end{equation*}
is a vector of coefficients from the scaling equation (\ref{okrajova_skalova_rce}).
The matrix $\mathbf{M}_R$ is obtained from a matrix
$\mathbf{M}_L$ by reversing the ordering of rows.

It follows from (\ref{waveletova_rce}) that the matrix $\mathbf{M}_{j,1}$ is of the size
$2^{j+1} \times 2^j$ and has the structure
\begin{equation} \label{waveletova_mce}
\mathbf{M}_{j,1}=   
\frac{1}{\sqrt{2}} \left[ \begin{array}{c c c c c c c c c c}
- \frac{1}{2} & \frac{1}{2} & 0& 0 & 0 & 0 & 0 & \ldots & 0 & 0 \\ \vspace{1mm}
 0 & 0 & -\frac{1}{2} & \frac{1}{2} & 0 & 0 & 0 &  \ldots & 0 & 0 \\ \vspace{1mm}
 0 & 0 & 0& 0 & -\frac{1}{2} & \frac{1}{2} &  & & 0 & 0\\ \vspace{1mm}
 \vdots & \vdots &  & & & & & &  & \vdots\\
 0 & 0 & \ldots & 0 & 0 & 0 & -\frac{1}{2} & \frac{1}{2} & 0 & 0\\ \vspace{1mm}
 0 & 0 & \ldots & 0 & 0 & 0 & 0 & 0 & -\frac{1}{2} & \frac{1}{2}
\end{array} \right]^T.
\end{equation}
%
The following lemmas are crucial for the proof of a Riesz basis property.

\begin{lem} \label{explicit_Mj0} 
Let $j \geq 2$ and the entries $\tilde{M}_{k,l}^{j,0}$, $k \in \mathcal{I}_{j+1}$, $l \in \mathcal{I}_{j}$, of the matrix $\tilde{\mathbf{M}}_{j,0}$ be given by:
\begin{eqnarray} \label{mkl_explicit}
\tilde{M}^{j,0}_{2,l} &=& \tilde{M}^{j,0}_{1,l} = \frac{ d^j_1}{a^{\left|1-l\right|}}+\frac{ d^j_n}{a^{\left|n-l\right|}}, \\
\nonumber \tilde{M}^{j,0}_{2^{j+1},l} &=& \tilde{M}^{j,0}_{2^{j+1} - 1, l} = \frac{ d^j_1}{a^{\left|n-l\right|}}+\frac{ d^j_n}{a^{\left|1-l\right|}},
\end{eqnarray}
where $n=2^j$, $a=-3-2\sqrt{2}$, 
\begin{eqnarray} \label{def_djk}
d^j_1 &=& \frac{ 6 \alpha_n}{3 + \sqrt{2} }, \quad d^j_n = \frac{ -36 b \, \alpha_n \, a^{2-n} }{11+ 6 \sqrt{2} }, \\
\nonumber \alpha_n &=& \left( 1-\frac{36 \, b^2 \, a^{4-2n} }{11 + 6 \sqrt{2} } \right)^{-1}, \quad
 b=\frac{13-9 \sqrt{2}}{6},
\end{eqnarray}
and for  $k=2, \ldots, n-1$ and $l \in \mathcal{I}_j$ let
\begin{equation} \label{mkl_explicit2}
\tilde{M}^{j,0}_{2k,l} \! = \! \tilde{M}^{j,0}_{2k-1,l} = \frac{ 1 }{a^{\left|k-l\right|}} \! + \frac{ d^j_k}{a^{\left|1-l\right|}} + \frac{  d^j_{n+1-k}}{a^{\left|n-l\right|}},
\end{equation}
where 
\begin{equation} \label{def_djk_2}
 d^j_k = \frac{ -6b  \, \alpha_n \, a^{2-k} }{ 3+ \sqrt{2}}- \frac{ 36 b \, \alpha_n \, a^{k+3-2n}}{ 11 + 6 \sqrt{2}}. 
\end{equation}
Then \begin{equation} \label{algbior1}
\mathbf{M}_{j,0}^T \tilde{\mathbf{M}}_{j,0} = \mathbf{I}_j, \quad and \quad
\mathbf{M}_{j,1}^T \tilde{\mathbf{M}}_{j,0} = \mathbf{0}_j,
\end{equation}
where $\mathbf{I}_j$ denotes the identity matrix and $\mathbf{0}_j$ denotes the zero matrix of the appropriate size.
\end{lem}

\begin{proof}
By similar approach as in \cite{Cerna2014a, Cerna2014b} we derive the explicit form of the entries $\tilde{M}_{k,l}^{j,0}$, $k \in \mathcal{I}_{j+1}$, $l \in \mathcal{I}_{j}$, of the matrix $\tilde{\mathbf{M}}_{j,0}$ such that $(\ref{algbior1})$ is satisfied. From $(\ref{waveletova_mce})$ we obtain 
\begin{equation} \label{mkl}
\tilde{M}_{2k-1,l}= \tilde{M}_{2k,l}, \quad {\rm for} \, k=1, \ldots, 2^{j}.
\end{equation}

We substitute (\ref{mkl}) into $(\ref{algbior1})$ and we obtain
a new system $\mathbf{A}_j \mathbf{B}_j = \mathbf{I}_j$, where
\begin{equation*} 
\mathbf{A}_j=\frac{1}{\sqrt{2}} \left[ \begin{array}{c c c c c c }
 \frac{13}{8} & \frac{3}{8} & 0& \ldots & & 0 \\ \vspace{1mm}
 \frac{1}{4} & \frac{3}{2} & \frac{1}{4} & & & \vdots \\ \vspace{1mm}
 0 & \frac{1}{4} & \frac{3}{2}& \frac{1}{4} & & 0 \\ 
 \vdots & & \ddots & \ddots & \ddots & \\ \vspace{1mm} 
 0 &  & & \frac{1}{4} & \frac{3}{2}  & \frac{1}{4} \\ 
 0 & & \ldots & 0& \frac{3}{8}  & \frac{13}{8} 
\end{array} \right] = \frac{\mathbf{H}_j}{\sqrt{2}} \left[ \begin{array}{c c c c c c }
 \frac{13}{12} & \frac{1}{4} & 0& \ldots & & 0 \\ \vspace{1mm}
 \frac{1}{4} & \frac{3}{2} & \frac{1}{4} & & & \vdots \\ \vspace{1mm}
 0 & \frac{1}{4} & \frac{3}{2}& \frac{1}{4} & & 0 \\ 
 \vdots & & \ddots & \ddots & \ddots & \\ \vspace{1mm} 
 0 &  & & \frac{1}{4} & \frac{3}{2}  & \frac{1}{4} \\ 
 0 & & \ldots & 0& \frac{1}{4}  & \frac{13}{12} 
\end{array} \right],
\end{equation*}
where 
\begin{equation*}
\left( \mathbf{H}_j \right)_{k,l} = \begin{cases} \frac{3}{2}, & \left(k,l\right)=\left(1,1\right), \,  \left(k,l\right)=\left(2^j,2^j\right) \\
1, &  k=l, k \neq 1, k \neq 2^j, \\
0, & {\rm otherwise},
\end{cases}
\end{equation*}
and $\mathbf{B}_j$ is the $2^j \times 2^j$ matrix with entries
$B^j_{k,l}=\tilde{M}^{j,0}_{2k,l}$, $k, l \in \mathcal{I}_j$. 
We factorize the matrix $\mathbf{A}_j$ as $\mathbf{A}_j=\mathbf{H}_j \mathbf{C}_j \mathbf{D}_j$, where
\begin{equation*}
\mathbf{C}_j=\frac{1}{\sqrt{2}} \left[ \begin{array}{c c c c c c }
 \frac{3+2\sqrt{2}}{4} & \frac{1}{4} & 0& 0 & \ldots & 0 \\ \vspace{1mm}
 \frac{1}{4} & \frac{3}{2} & \frac{1}{4} & & & \vdots\\ \vspace{1mm}
 0 & \frac{1}{4} & \frac{3}{2}& \frac{1}{4} & & 0\\
 \vdots & & \ddots & \ddots & \ddots & \\ \vspace{1mm}
 0 & & & \frac{1}{4} & \frac{3}{2}  & \frac{1}{4}\\ \vspace{1mm}
  0 & \ldots & 0 & 0& \frac{1}{4}  & \frac{3+2\sqrt{2}}{4}\\ 
\end{array} \right], 
\end{equation*}
and 
\begin{equation*}
\mathbf{D}_j= \left[ \begin{array}{c c c c c c c }
 \frac{3+\sqrt{2}}{ 6} & 0 & 0 & \ldots & 0 & 0 & \frac{b}{ a^{n-2} } \\ \vspace{1mm}
 b & 1 &  0& & 0 & 0 & \frac{b}{ a^{n-3} } \\ \vspace{1mm}
 \frac{b}{a } & 0 & 1 &   & 0 & 0&  \frac{b}{ a^{n-4} }\\\vspace{1mm}
 \vdots & \vdots & & \ddots &  & \vdots & \vdots \\ \vspace{1mm}
 \frac{b}{ a^{n-4} }  & 0 & 0 & & 1 & 0 & \frac{b}{a} \\ \vspace{1mm}
 \frac{b}{ a^{n-3} }  & 0 & 0 & & 0 & 1 & b \\ \vspace{1mm}
  \frac{b}{ a^{n-2} } & 0 & 0 & \ldots & 0  & 0 & \frac{3 + \sqrt{2} }{6}\\ 
\end{array} \right],
\end{equation*}
More precisely, the entries $D^j_{k,l}$ of the matrix $\mathbf{D}_j$ are given by:
\begin{eqnarray*}
D^j_{1,1} &=& D^j_{n,n}= \frac{3 + \sqrt{2}}{ 6 }, \\
\nonumber D^j_{k,1} &=& D^j_{n+1-k,n} = \frac{b}{ a^{k-2}}, \quad {\rm for} \, k=2, \ldots, n, \\
\nonumber D^j_{k,k} &=& 1, \quad {\rm for} \, k=2, \ldots, n-1, \\
\nonumber D^j_{k,l} &=& 0, \quad {\rm otherwise}.
\end{eqnarray*}
It is easy to verify that $\tilde{\mathbf{C}}_j=\mathbf{C}_j^{-1}$ has entries
$\tilde{C}^j_{k,l}= a^{ - \left| k-l \right|}$,
and 
the matrix $\mathbf{D}_j^{-1}$ has the structure:
\begin{equation*}
\mathbf{D}_j^{-1}= \left[ \begin{array}{c c c c c }
 d^j_1 & 0 & \ldots & 0 & d^j_n  \\
 d^j_2 & 1 &  & 0 & d^j_{n-1} \\
 \vdots & & \ddots &  & \vdots \\
 d^j_{n-1} & 0 &  & 1 & d^j_2 \\
  d^j_n & 0 & \ldots & 0  & d^j_1 \\
\end{array} \right],
\end{equation*}
with $d^j_k$ given by (\ref{def_djk}) nad (\ref{def_djk_2}). 
Since the matrices $\mathbf{C}_j$, $\mathbf{D}_j$ and $\mathbf{H}_j$ are invertible, we can define 
\begin{equation} \label{def_bj}
\mathbf{B}_j = \mathbf{A}_j^{-1}=\mathbf{D}_j^{-1} \mathbf{C}_j^{-1} \mathbf{H}_j^{-1} . 
\end{equation}
Substituting (\ref{def_bj}) into $(\ref{mkl})$ the lemma is proved. 
$\qquad$ \end{proof}

\begin{lem}
There exist unique matrices $\tilde{\mathbf{M}}_{j,1}$, $j \geq 2$, such that
\begin{equation} \label{algbior2}
\mathbf{M}_{j,0}^T \tilde{\mathbf{M}}_{j,1} = \mathbf{0}_j, \quad and \quad
\mathbf{M}_{j,1}^T \tilde{\mathbf{M}}_{j,1} = \mathbf{I}_j.
\end{equation}
\end{lem}

\begin{proof}
For $l \in \mathcal{I}_{j+1}$ and $k \in \mathcal{I}_{j}$ the entries $\tilde{M}_{k,l}^{j,1}$ of the matrix $\tilde{\mathbf{M}}_{j,1}$ satisfy
\begin{equation*} 
\tilde{M}^{j,1}_{2k-1,l} = 2 \delta_{2k-1,2l-1} + \tilde{M}^{j,1}_{2k,l}.
\end{equation*}
Using these relations we obtain a system of equations with the matrix $\mathbf{A}_j$ defined in the proof of Lemma~\ref{norm_Mj0}.
Since the matrix $\mathbf{A}_j$ is invertible, the matrix $\tilde{\mathbf{M}}_{j,1}$ exists and is unique. 
$\qquad$ \end{proof}

\begin{lem} \label{decomposition_Phij}
We have $\Phi_{j+1}=\tilde{\mathbf{M}}_{j,0} \Phi_j + \tilde{\mathbf{M}}_{j,1} \Psi_j$ for all $j \geq 2$. 
\end{lem}

\begin{proof}
Due to (\ref{ref_matrices}) we have
\begin{equation*}
\begin{bmatrix} \Phi_j \\ \Psi_j \end{bmatrix} = \begin{bmatrix} \mathbf{M}^T_{j,0} \vspace{1mm} \\ \mathbf{M}^T_{j,1} \end{bmatrix} \Phi_{j+1}, \quad j \geq 2.
\end{equation*}
Multiplying this equation by the matrix $\left[ \tilde{\mathbf{M}}_{j,0}, \tilde{\mathbf{M}}_{j,1} \right]$ from the left-hand side and using (\ref{algbior1}) and (\ref{algbior2}) the lemma is proved.
$\qquad$ \end{proof}

For any matrix $\mathbf{M}$ of the size $m \times n$ we set
\begin{equation*} 
\left\|\mathbf{M}\right\|_2 = \sup_{\mathbf{v} \in \mathbb{R}^n, \mathbf{v} \neq \mathbf{0} } \frac{ \left\| \mathbf{M} \mathbf{v} \right\|_2}{ \left\| \mathbf{v} \right\|_2}
\end{equation*}
and 
\begin{equation*} 
\left\| \mathbf{M} \right\|_1 = \max\limits_{l=1, \ldots, n} \sum\limits_{k=1}^m \left| M_{k,l} \right|, \quad \left\| \mathbf{M} \right\|_{\infty} = \max\limits_{k=1, \ldots, m} \sum\limits_{l=1}^n \left| M_{k,l} \right|.
\end{equation*}
It is well-known that 
\begin{equation} \label{estimate_2_norm}
\left\| \mathbf{M} \right\|_2 \leq \sqrt{ \left\| \mathbf{M} \right\|_1 \left\| \mathbf{M} \right\|_{\infty} }.
\end{equation}

\begin{lem} \label{norm_Mj0}
The matrices $\tilde{\mathbf{M}}_{j,0}$, $j \geq 2$, have uniformly bounded norms, i.e. there exists $C \in \mathbb{R}$ independent of $j$ such that
$\left\|\tilde{\mathbf{M}}_{j,0} \right\|_2 \leq C$ for all $j \geq 2$.
\end{lem}

\begin{proof}
Let $\mathbf{B}_j$, $\mathbf{C}_j$, $\mathbf{D}_j$ and $\mathbf{H}_j$ are the same as in the proof of Lemma~\ref{norm_Mj0}.
From (\ref{mkl_explicit}) and (\ref{mkl_explicit2}) we have
$\tilde{\mathbf{M}}_{j,0} = \mathbf{G}_j \mathbf{B}_j$, where $\mathbf{G}_j$ is of the size $2^{j+1} \times 2^j$
with entries
\begin{equation*}
\left( \mathbf{G}_j \right)_{k,l} = \begin{cases} 
1, & k=2l-1 \ {\rm or} \ k=2l , \\
0, & {\rm otherwise}.
\end{cases}
\end{equation*}
Therefore $\tilde{\mathbf{M}}_{j,0} = \mathbf{G}_j \mathbf{B}_j= \mathbf{G}_j \mathbf{D}_j^{-1} \mathbf{C}_j^{-1} \mathbf{H}_j^{-1}$. Since the matrices $\mathbf{G}_j$, $\mathbf{C}_j^{-1}$, $\mathbf{D}_j^{-1}$ and $\mathbf{H}_j^{-1}$ are given by simple explicit formulas, the formula for the sum of a geometric series and (\ref{estimate_2_norm}) yield 
\begin{equation*}
\left\| \mathbf{G}_j \right\|_2 \leq \sqrt{2}, \quad \left\| \mathbf{C}_j^{-1} \right\|_2 \leq \sqrt{2}, \quad \left\| \mathbf{D}_j^{-1} \right\|_2 \leq 1.4 \quad {\rm and}
\quad \left\| \mathbf{H}_j^{-1} \right\|_2 \leq 1,
\end{equation*}
and thus $\left\| \tilde{\mathbf{M}}_{j,0} \right\|_2 \leq 2.8$ for all $j \geq 2$.
$\qquad$ \end{proof}

\begin{lem} \label{norm_tildeS_j}
Let $\mathbf{S}_j=\tilde{\mathbf{M}}^T_{j,0} \tilde{\mathbf{M}}^T_{j+1,0}$, $j \geq 3$,
and $\tilde{\mathbf{S}}_j$ be the matrix given by
\begin{equation*} 
\left( \tilde{\mathbf{S}}_j \right)_{k,l}=\left( \mathbf{S}_j \right)_{2k-1,l}+\left( \mathbf{S}_j \right)_{2k,l}, \quad
k \in \mathcal{I}_{j-1}, l \in \mathcal{I}_{j+2}.
\end{equation*}
Then there exists a constant $C$ independet of $j$ such that $\left\| \tilde{\mathbf{S}}_j \right\|_2 < C < 2 \sqrt{2}$.
\end{lem}

{\em Proof}.
Let $\mathbf{K}_j$ be a $2^j \times 2^{j+1}$ matrix with entries
\begin{equation} \label{definition_Kj}
\left( \mathbf{K}_j \right)_{k,2l-1}=\left( \mathbf{K}_j \right)_{k,2l}=a^{- \left| k - l  \right|}, \quad k, l \in \mathcal{I}_{j}, \quad a=-3-2 \sqrt{2},
\end{equation}
and let $\mathbf{L}_j=\tilde{\mathbf{M}}_{j,0}^{T}-\mathbf{K}_j$.
We know the explicit expression of the matrix $\mathbf{L}_j$, because the explicit expressions of both $\tilde{\mathbf{M}}_{j,0}$ and $\mathbf{K}_j$ are known. We have
\begin{equation*}
\mathbf{S}_j = \tilde{\mathbf{M}}_{j,0}^T \tilde{\mathbf{M}}_{j+1,0}^T = \mathbf{K}_j \mathbf{K}_{j+1} + \mathbf{K}_j \mathbf{L}_{j+1} +
\mathbf{L}_j \mathbf{K}_{j+1} + \mathbf{L}_j \mathbf{L}_{j+1}.
\end{equation*}
Let us denote
\begin{equation*}
\mathbf{N}_j=\mathbf{K}_j \mathbf{K}_{j+1}, \mathbf{O}_j=\mathbf{K}_j \mathbf{L}_{j+1}, 
\mathbf{P}_j=\mathbf{L}_j \mathbf{K}_{j+1}, \mathbf{Q}_j=\mathbf{L}_j \mathbf{L}_{j+1},
\end{equation*}
and $\tilde{\mathbf{N}}_j$, $\tilde{\mathbf{O}}_j$, $\tilde{\mathbf{P}}_j$, and $\tilde{\mathbf{Q}}_j$
be derived from $\mathbf{N}_j$, $\mathbf{O}_j$, $\mathbf{P}_j$ and $\mathbf{Q}_j$ by similar way as $\tilde{\mathbf{S}}_j$
from $\mathbf{S}_j$.
Then $\tilde{\mathbf{S}}_j = \tilde{\mathbf{N}}_j + \tilde{\mathbf{O}}_j + \tilde{\mathbf{P}}_j + \tilde{\mathbf{Q}}_j.$
From (\ref{definition_Kj}) we have 
for $k \in \mathcal{I}_{j}$, $l \in \mathcal{I}_{j+1}$
\begin{equation*}
\left( \mathbf{N}_j \right)_{k,2l-1}=\left(\mathbf{N}_j\right)_{k,2l} = \mathbf{u}_k^T \mathbf{v}_l,
\end{equation*}  
where
\begin{eqnarray*}
\mathbf{u}_k &=& \left[ \frac{1}{a^{k-1}}, \frac{1}{a^{k-1}}, \frac{1}{a^{k-2}}, \ldots, \frac{1}{a}, \frac{1}{a}, 1, 1,
\frac{1}{a}, \frac{1}{a}, \ldots ,\frac{1}{a^{n-k}}, \frac{1}{a^{n-k}} \right]^T, \\
\mathbf{v}_l &=& \left[ \frac{1}{a^{l-1}} , \frac{1}{a^{l-2}}, \ldots, \frac{1}{a}, 1, \frac{1}{a}, \ldots, \frac{1}{a^{2n-l}}
\right]^T,
\end{eqnarray*}
$n=2^j$. Due to the structure of the vector $\mathbf{u}_k$ we can write
\begin{equation*}
\left( \mathbf{N}_j \right)_{k,l} = \frac{a+1}{a} \, \tilde{\mathbf{u}}_k^T \tilde{\mathbf{v}}_l,  
\end{equation*}
where
\begin{eqnarray*}
\tilde{\mathbf{u}}_k &=& \left[ \frac{1}{a^{k-1}}, \frac{1}{a^{k-2}}, \ldots, \frac{1}{a}, 1, \frac{1}{a}, \ldots , \frac{1}{a^{n-k}} \right]^T, \\
\tilde{\mathbf{v}}_l &=& \begin{cases} \left[ \frac{1}{a^{l-2}} , \frac{1}{a^{l-4}}, \ldots, \frac{1}{a^2}, 1, \frac{1}{a}, \frac{1}{a^3} \ldots, \frac{1}{a^{2n-l-1}} \right]^T, & \quad l \, \, \textrm{even}, \\
\left[ \frac{1}{a^{l-2}} , \frac{1}{a^{l-4}}, \ldots, \frac{1}{a}, 1, \frac{1}{a^2}, \frac{1}{a^4} \ldots, \frac{1}{a^{2n-l-1}} \right]^T, & \quad l \, \,  \textrm{odd}.
\end{cases}
\end{eqnarray*}
For $k > \frac{l}{2}$, $l \in \mathcal{I}_{j+1}$, $l$ even, we have
\begin{eqnarray*}
\left(\mathbf{N}_j\right)_{k,2l} &=& \frac{a+1}{a} \left( \sum_{m=1}^{\frac{l}{2}} a^{3m-k-l} + \sum_{m=\frac{l}{2}+1}^{k} a^{l+1-k-m} +
\sum_{m=k +1}^{n} a^{l+k+1-3m}  \right) \\
\nonumber &=& \frac{a+1}{a} \left( a^{\frac{l}{2}-k} \frac{1-\left(\frac{1}{a^3}\right)^{\frac{l}{2}} }{1-\frac{1}{a^3}} +
a^{\frac{l}{2}-k} \frac{1-\left(\frac{1}{a}\right)^{k-\frac{l}{2} }}{1-\frac{1}{a}} + a^{l-2-2k} \frac{1-\left(\frac{1}{a^3}\right)^{n-k } }{1-\frac{1}{a^3}} \right).
\end{eqnarray*}
Similarly for $k > \frac{l-1}{2}$, $l \in \mathcal{I}_{j+1}$, $l$ odd, we obtain
\begin{eqnarray*}
\!  \! \left(\mathbf{N}_j\right)_{k,2l} \! &=& \! \frac{a \! + \! 1}{a} \! \left( \sum_{m=1}^{\frac{l-1}{2}} \!  a^{3m-k-l} + \sum_{m=\frac{l+1}{2}}^{k} \! \! \!  a^{l+1-k-m} +
\sum_{m=k+1}^{n} \! \! \! a^{l+k+1-3m}  \right) \\
\nonumber \! &=& \! \frac{a \! + \! 1}{a} \! \left( \! a^{\frac{l}{2}-k-\frac{3}{2}} \frac{1 \!  - \! \left(\frac{1}{a^3}\right)^{\frac{l-1}{2}} }{1-\frac{1}{a^3}} +
a^{\frac{l+1}{2}-k} \frac{1 \! - \! \left(\frac{1}{a^3}\right)^{k-\frac{l-1}{2}}}{1-\frac{1}{a}} + a^{l-2-2k} \frac{1 \! - \! \left(\frac{1}{a^3}\right)^{n-k}}{1-\frac{1}{a^3}} \right) \! . 
\end{eqnarray*}
If $k \leq \frac{l}{2}$, $l \in \mathcal{I}_{j+1}$, $l$ even, then we have
\begin{eqnarray*}
\!  \! \left(\mathbf{N}_j\right)_{k,2l} &=& \frac{a+1}{a} \left( \sum_{m=1}^{k} a^{3m-k-l} + \sum_{k+1}^{\frac{l}{2}} a^{m+k-l} +
\sum_{m=\frac{l}{2}+1}^{n} a^{l+k+1-3m}  \right) \\
\nonumber &=& \frac{a+1}{a} \left( a^{2k-l} \frac{1-\left(\frac{1}{a^3}\right)^{k}}{1-\frac{1}{a^3}} +
a^{k-\frac{l}{2}} \frac{1-\left(\frac{1}{a}\right)^{\frac{l}{2}-k}}{1-\frac{1}{a}} + a^{k-\frac{l}{2}-2} \frac{1-\left(\frac{1}{a^3}\right)^{n-\frac{l}{2}} }{1-\frac{1}{a^3}} 
\right).
\end{eqnarray*}
If $k \leq \frac{l-1}{2}$, $l \in \mathcal{I}_{j+1}$, $l$ odd, then we have
\begin{eqnarray*}
\!  \!  \left(\mathbf{N}_j\right)_{k,2l} \! \! &=& \!  \!  \frac{a \! + \! 1}{a} \left( \sum_{m=1}^{k} a^{3m-k-l} + \sum_{k+1}^{\frac{l-1}{2}} a^{m+k+2-l} +
\sum_{m=\frac{l+1}{2}+1}^{n} a^{l+k+1-3m}  \right) \\
\nonumber \!  \! &=& \!  \! \frac{a \! + \! 1}{a} \! \left( a^{ 2k  -  l } \frac{1-\left(\frac{1}{a^3}\right)^{k}}{1-\frac{1}{a^3}} +
a^{k-\frac{l}{2}-\frac{1}{2}} \frac{1-\left(\frac{1}{a}\right)^{\frac{l-1}{2}-k}}{1-\frac{1}{a}} + a^{k-\frac{l}{2}-\frac{1}{2}} \frac{1-\left(\frac{1}{a^3}\right)^{n-\frac{l-1}{2}} }{1-\frac{1}{a^3}} \right).
\end{eqnarray*}
To compute an upper bound for the norm of the matrix $\tilde{\mathbf{S}}_j $, we compute bounds for 
the sums of absolute values of entries in rows and columns for matrices $\tilde{\mathbf{N}}_j$, $\tilde{\mathbf{O}}_j$, $\tilde{\mathbf{P}}_j$, and $\tilde{\mathbf{Q}}_j$. Since the values in columns of the matrix $\tilde{\mathbf{N}}_j$
are exponentially decreasing, we can compute several largest values in each column and estimate the sum of absolute values of the remaining entries.
We denote
\begin{equation*}
\bar{\mathcal{I}}_{j+2} =  \left\{1,2,3,4, 2^{j+2}-3, 2^{j+2}-2, 2^{j+2}-1, 2^{j+2} \right\}, \quad
\nonumber \check{\mathcal{I}}_{j+2} = \mathcal{I}_{j+2} \backslash \bar{\mathcal{I}}_{j+2}
\end{equation*}
and we set
\begin{equation*}
\left( \tilde{\mathbf{N}}_j \right)_{k,l}=0, \quad {\rm for} \ k \notin \mathcal{I}_{j-1}.
\end{equation*}
For $l$ such that $l \! \! \mod 8 \in \left\{ 0, 1, 6, 7\right\}$ and $l \in \check{\mathcal{I}}_{j+2}$ we obtain
\begin{eqnarray*}
\!  \! \sum_{k=1}^{2^{j-1}} \left| \left( \tilde{\mathbf{N}}_j \right)_{k,l} \right| & \leq & 
\left| \left( \tilde{\mathbf{N}}_j \right)_{\left\lfloor \frac{l}{8} \right\rfloor - 1,l} \right| +
\left| \left( \tilde{\mathbf{N}}_j \right)_{\left\lfloor \frac{l}{8} \right\rfloor,l} \right| +
\left| \left( \tilde{\mathbf{N}}_j \right)_{\left\lfloor \frac{l}{8} \right\rfloor +1,l} \right| \\
\nonumber &+& \sum\limits_{k=1}^{\left\lfloor \frac{l}{8} \right\rfloor - 2} \left| \left( \tilde{\mathbf{N}}_j \right)_{k,l} \right| 
+ \sum\limits_{k=\left\lfloor \frac{l}{8} \right\rfloor + 2}^{2^{j-1}} \left| \left( \tilde{\mathbf{N}}_j \right)_{k,l} \right| \\
\nonumber & \leq & 0.018 + 0.727 + 0.239 + 0.007 + 0.001 \leq 1.
\end{eqnarray*}
For $l$ such that $l \! \! \mod 8 \in \left\{ 2, 3, 4, 5\right\}$ and $l \in \check{\mathcal{I}}_{j+2}$ we obtain
\begin{eqnarray*}
\sum_{k=1}^{2^{j-1}} \left| \left( \tilde{\mathbf{N}}_j \right)_{k,l} \right| & \leq & 
\left| \left( \tilde{\mathbf{N}}_j \right)_{\left\lfloor \frac{l}{8} \right\rfloor - 1,l} \right| +
\left| \left( \tilde{\mathbf{N}}_j \right)_{\left\lfloor \frac{l}{8} \right\rfloor,l} \right| +
\left| \left( \tilde{\mathbf{N}}_j \right)_{\left\lfloor \frac{l}{8} \right\rfloor +1,l} \right| \\
\nonumber &+& \sum\limits_{k=1}^{\left\lfloor \frac{l}{8} \right\rfloor - 2} \left| \left( \tilde{\mathbf{N}}_j \right)_{k,l} \right| 
+ \sum\limits_{k=\left\lfloor \frac{l}{8} \right\rfloor + 2}^{2^{j-1}} \left| \left( \tilde{\mathbf{N}}_j \right)_{k,l} \right| \\
\nonumber & \leq & 0.101 + 0.566 + 0.037 + 0.002 + 0.004 \leq 1.
\end{eqnarray*}
For $l \in \bar{\mathcal{I}}_{j+2}$ we have
\begin{equation*}
\sum_{k=1}^{2^{j-1}} \left| \left( \tilde{\mathbf{N}}_j \right)_{k,l} \right| \leq 0.5.
\end{equation*}

We use the similar approach for computing the sums of absolute values of the entries in rows.
We obtain 
\begin{equation*}
\sum_{k=1}^{2^{j-1}} \left| \left( \tilde{\mathbf{N}}_j \right)_{k,l} \right| \leq
\begin{cases}
0.73, & l \in \bar{\mathcal{I}}_{j+2}, \\
1.00, & l \in \check{\mathcal{I}}_{j+2},
\end{cases} \quad
\sum_{l=1}^{2^{j+2}} \left| \left( \tilde{\mathbf{N}}_j \right)_{k,l} \right| \leq 
\begin{cases}
5.95, & k=1, 2^{j-1}, \\
6.80, & {\rm otherwise}.
\end{cases}
\end{equation*}
Similarly, we obtain
\begin{equation*}
\sum_{k=1}^{2^{j-1}} \left| \left( \tilde{\mathbf{O}}_j  \right)_{k,l} \right| \leq
\begin{cases}
0.13, & l \in \bar{\mathcal{I}}_{j+2}, \\
0.04, & l \in \check{\mathcal{I}}_{j+2},
\end{cases} \quad
\sum_{l=1}^{2^{j+2}} \left| \left( \tilde{\mathbf{O}}_j  \right)_{k,l} \right| \leq 
\begin{cases}
0.30, & k=1, 2^{j-1}, \\
0.02, & {\rm otherwise},
\end{cases}
\end{equation*}
\begin{equation*}
\sum_{k=1}^{2^{j-1}} \left| \left( \tilde{\mathbf{P}}_j  \right)_{k,l} \right| \leq
\begin{cases}
0.15, & l \in \bar{\mathcal{I}}_{j+2}, \\
0.05, & l \in \check{\mathcal{I}}_{j+2},
\end{cases} \quad
\sum_{l=1}^{2^{j+2}} \left| \left( \tilde{\mathbf{P}}_j \right)_{k,l} \right| \leq 
\begin{cases}
0.68, & k=1, 2^{j-1}, \\
0.04, & {\rm otherwise},
\end{cases}
\end{equation*}
\begin{equation*}
\sum_{k=1}^{2^{j-1}} \left| \left( \tilde{\mathbf{Q}}_j  \right)_{k,l} \right| \leq
\begin{cases}
0.03, & l \in \bar{\mathcal{I}}_{j+2}, \\
0.01, & l \in \check{\mathcal{I}}_{j+2},
\end{cases} \quad
\sum_{l=1}^{2^{j+2}} \left| \left( \tilde{\mathbf{Q}}_j \right)_{k,l} \right| \leq 
\begin{cases}
0.06, & k=1, 2^{j-1}, \\
0.01, & {\rm otherwise}.
\end{cases}
\end{equation*}
Therefore using (\ref{estimate_2_norm}) we have
\begin{equation*}
\left\| \tilde{\mathbf{S}}_j \right\|_2 \leq \sqrt{1.1 \cdot 7} < 2 \sqrt{2}.
\qquad\endproof
\end{equation*}

\begin{lem} \label{norm_nasobeni_MjMj1} 
Let $m,n \geq 2$, $m < n$, then there exists a constant $C<2$ such that
\begin{equation*}
\left\| \tilde{\mathbf{M}}^T_{m,0} \tilde{\mathbf{M}}^T_{m+1,0} \ldots \tilde{\mathbf{M}}^T_{n,0} \tilde{\mathbf{M}}^T_{n+1,0}\right\|_2
\leq C \, \left\| \tilde{\mathbf{M}}^T_{m,0} \tilde{\mathbf{M}}^T_{m+1,0} \ldots \tilde{\mathbf{M}}^T_{n-1,0} \right\|_2.
\end{equation*}
\end{lem}

\begin{proof}
For $m$ and $n$ fixed such that $m,n \geq 2$, $m  < n$, we use notation:
\begin{equation*}
\mathbf{R}=\tilde{\mathbf{M}}^T_{m,0} \tilde{\mathbf{M}}^T_{m+1,0} \ldots \tilde{\mathbf{M}}^T_{n-1,0}, \quad
\mathbf{S}=\tilde{\mathbf{M}}^T_{n,0} \tilde{\mathbf{M}}^T_{n+1,0}.
\end{equation*}
Due to the structure of the matrices $\tilde{\mathbf{M}}_{j,0}$ given in Lemma~\ref{norm_Mj0}
we have 
\begin{equation*}
\mathbf{R}_{k,2l}=\mathbf{R}_{k,2l-1}, \quad k \in \mathcal{I}_m, l \in \mathcal{I}_{n-1}. 
\end{equation*}
Therefore, we can write 
$\mathbf{R} \mathbf{S} = \tilde{\mathbf{R}} \tilde{\mathbf{S}}$,
where the matrix $\tilde{\mathbf{R}}$ is $2^m \times 2^{n-1}$ matrix containing even columns of the matrix $\mathbf{R}$, i.e.
$\tilde{\mathbf{R}}_{k,l}=\mathbf{R}_{k,2l}$, and the matrix $\tilde{\mathbf{S}} $ is given by
\begin{equation*} 
\tilde{\mathbf{S}}_{k,l}= \mathbf{S}_{2k-1,l}+ \mathbf{S}_{2k,l}, \quad k \in \mathcal{I}_{n-1}, l \in \mathcal{I}_{n+2}.
\end{equation*}
We have
\begin{eqnarray*}
\left\| \tilde{\mathbf{R}} \right\|_2 &=& \sup\limits_{\mathbf{x} \in \mathbb{R}, \mathbf{x} \neq 0  } \frac{\left\| \tilde{\mathbf{R}} \mathbf{x} \right\|_2}{ \left\| \mathbf{x} \right\|_2} = \sup\limits_{\mathbf{x} \in \mathbb{R}, \mathbf{x} \neq 0} \frac{ \left( \sum\limits_{k \in \mathcal{I}_m} \left( \sum\limits_{l \in \mathcal{I}_{n-1} } \tilde{\mathbf{R}}_{k,l} \mathbf{x}_l\right)^2 \right)^{1/2}}{\left\| \mathbf{x} \right\|_2}.
\end{eqnarray*}
Let $\tilde{\mathbf{x}}$ be a vector of the length $q=2^n$ such that $\tilde{x}_{2j-1}=\tilde{x}_{2j}=x_j$ and let 
\begin{equation*}
\tilde{X}=\left\{ \tilde{\mathbf{x}} \in \mathbb{R}^q: \tilde{x}_{2j-1}=\tilde{x}_{2j}, \tilde{\mathbf{x}} \neq \mathbf{0} \right\}.
\end{equation*}
Then
$\left\| \tilde{\mathbf{x}} \right\|=\sqrt{2} \left\| \mathbf{x} \right\|$ and we have
\begin{eqnarray*}
\left\| \tilde{\mathbf{R}} \right\|_2 &=& \sup\limits_{\tilde{\mathbf{x}} \in \tilde{X}} \frac{ \left(\sum\limits_{k \in \mathcal{I}_m} \left(\sum\limits_{l \in \mathcal{I}_n} 2^{-1} \mathbf{R}_{k,l} \tilde{\mathbf{x}}_l\right)^2 \right)^{1/2}}{2^{-1/2} \left\| \tilde{\mathbf{x}} \right\|_2}\\
\nonumber &\leq& \sup\limits_{\tilde{\mathbf{x}} \in \mathbb{R}^q, \tilde{\mathbf{x}} \neq 0 } \frac{ 2^{-1} \left(\sum\limits_{k \in \mathcal{I}_m} \left(\sum\limits_{l \in \mathcal{I}_n} \mathbf{R}_{k,l} \tilde{\mathbf{x}}_l\right)^2 \right)^{1/2}}{2^{-1/2} \left\| \tilde{\mathbf{x}} \right\|_2} 
= \frac{\left\| \mathbf{R} \right\|_2}{\sqrt{2}}.
\end{eqnarray*}
Using Lemma~\ref{norm_tildeS_j} we obtain
\begin{equation*}
\left\| \mathbf{R} \mathbf{S} \right\|_2 = \left\| \tilde{\mathbf{R}} \tilde{\mathbf{S}} \right\|_2 \leq \left\| \tilde{\mathbf{R}} \right\|_2 \left\| \tilde{\mathbf{S}} \right\|_2 \leq C \left\| \mathbf{R} \right\|_2
\end{equation*}
with $C < 2$.
$\qquad$ \end{proof}

%

\begin{lem} \label{norm_nasobeni_Nj} 
There exist constants $C \in \mathbb{R}$ and $p < 0.5$ such that for all $m,n \geq 2$, $m<n$, we have
\begin{equation*} 
\left\| \tilde{\mathbf{M}}^T_{m,0} \tilde{\mathbf{M}}^T_{m+1,0} \ldots \tilde{\mathbf{M}}^T_{n-1,0} \right\|_2 \leq C \, 2^{p \left( n-m \right) }.
\end{equation*}
\end{lem}

\begin{proof}
The assertion of the lemma is a direct consequence of Lemma~\ref{norm_Mj0} and Lemma~\ref{norm_nasobeni_MjMj1}.
$\qquad$ \end{proof}

\section{Multivariate bases}

A basis on $\Omega_d=\left(0,1\right)^d$ is built from the univariate wavelet basis by a tensor product \cite{Urban2009}. Let $j \geq 2$, $\mathbf{k}=\left(k_1, \ldots, k_d \right)$, $\mathbf{k} \in \mathcal{I}_j^{d}:=\mathcal{I}_j \times \ldots \times \mathcal{I}_j$, and $\mathbf{x}=\left( x_1, \ldots, x_d \right) \in \Omega_d$.
We define the multivariate scaling functions by
\begin{equation*}
\phi_{j,\mathbf{k}}^{d} \left( \mathbf{x} \right)=\prod\limits_{l=1}^{d} \phi_{j,k_l} \left( x_l \right),
\end{equation*}
and for any $\mathbf{e}=\left(e_1,\ldots,e_d\right) \in E^d:=\left\{0,1\right\}^d \backslash \left( 0, \ldots, 0 \right)$, we define
the multivariate wavelet
\begin{equation*}
\psi_{j, \mathbf{e}, \mathbf{k}}^{d} \left(\mathbf{x}\right)=\prod\limits_{l=1}^d \psi_{j,e_l,k_l} \left(x_l\right),
\end{equation*}
where
\begin{equation*}
\psi_{j,e_l,k_l}=
\left\{ \begin{array}{ll}
\phi_{j,k_l}, & e_l=1, \\
\psi_{j,k_l}, & e_l=0. \\
\end{array} \right.
\end{equation*}
The basis on the unit cube $\Omega_d$ is then given by 
\begin{equation*}
\Psi^{dD}=\left\{\phi_{2, \mathbf{k}}^d,\ \mathbf{k} \in \mathcal{I}_j^{d} \right\} \cup \left\{\psi_{j, \mathbf{e}, \mathbf{k}}^{d}, \mathbf{e} \in E^d, \mathbf{k} \in \mathcal{I}_j^{d}, j \geq 2 \right\}.
\end{equation*}
By this approach, the regularity and polynomial exactness is preserved. 

\section{Riesz basis on Sobolev space} \label{sectionRB}
In this section, we prove that $\Psi$ is a Riesz basis of $H^1_0 \left( \Omega_1 \right)$ and $\Psi^{2D}$ is a Riesz basis of $H^1_0 \left( \Omega_2 \right)$. The proof is based on the lemmas from Section~\ref{section_refinement_matrices} and theory developed in \cite{Jia2011} that is summarized in the following theorem.
\begin{thm} \label{theorem_jia}
Let $H$ be a Hilbert space and let $V_j$, $j \geq J$,  be subspaces of $L_2 \left( \Omega \right)$ such that $V_j \subset V_{j+1}$ and $\cup_{j=J}^{\infty} V_j$ is dense in $H$. 
Let $H_q$ for fixed $q >0$ be a linear subspace of $H$ that is itself a normed linear space and assume that there exist positive constants $A_1$ and $A_2$ such that

a) If $f \in H_q$ has decomposition $f = \sum_{j \geq J} f_j$, $f_j \in V_j$ then
\begin{equation} \label{condition_Hq_1}
\left\| f \right\|_{H_q}^2 \leq A_1 \sum_{j \geq J} 2^{qj} \left\|f_j\right\|_H^2.
\end{equation}

b) For each $f \in H_q$ there exists a decomposition $f = \sum_{j \geq J} f_j$, $f_j \in V_j$, such that
\begin{equation} \label{condition_Hq_2}
\sum_{j \geq J} 2^{qj} \left\|f_j\right\|_H^2 \leq A_2 \left\| f \right\|_{H_q}^2.
\end{equation}
Furthermore, suppose that $P_j$ is a linear projection from $V_{j+1}$ onto $V_j$, 
$W_j$ is the kernel space of $P_j$, $\Phi_j=\left\{\phi_{j,k}, k\in \mathcal{I}_j \right\}$ are Riesz bases of $V_j$ with respect to the $L_2$--norm with uniformly bounded condition numbers and 
$\Psi_j=\left\{\psi_{j,k}, k\in \mathcal{I}_j \right\}$ are Riesz bases of $W_j$ with uniformly bounded condition numbers.
If there exist constants $C$ and $p$ such that $0<p<q$ and
\begin{equation} \label{norm_Pmn1}
\left\| P_m P_{m+1} \ldots P_{n-1} \right\| \leq C \, 2^{p \, \left( n-m \right)},
\end{equation}
then 
\begin{equation} \label{scaling_2naj}
\left\{2^{-Jq} \phi_{J,k}, k \in \mathcal{I}_J \right\} \cup \left\{2^{-jq} \psi_{j,k}, j \geq J, k \in \mathcal{I}_j \right\}
\end{equation}
is a Riesz basis of $H_q$.
\end{thm}

%
%
%

Now we define suitable projections $P_j$ from $V_{j+1}$ onto $V_j$ and show that these projections
satisfies (\ref{norm_Pmn1}). Then we show that $\Psi$ which differs from (\ref{scaling_2naj}) only by scaling is also a Riesz basis of $H_0^1 \left( 0, 1 \right)$. 
For $j \geq 2$ we define 
\begin{equation*}
\Gamma_j = \left\{ \phi_{j,k} \right\}_{k \in \mathcal{I}_j} \cup \left\{ \psi_{j,k} \right\}_{k \in \mathcal{I}_j}
\quad {\rm and} \quad
\mathbf{F}_j=\left\langle \Gamma_j, \Gamma_j \right\rangle.
\end{equation*}
Let a set 
\begin{equation} \label{def_hatphi_jk}
\hat{\Gamma}_j = \left\{ \hat{\phi}_{j,k} \right\}_{k \in \mathcal{I}_j} \cup \left\{ \hat{\psi}_{j,k} \right\}_{k \in \mathcal{J}_j}
\end{equation}
be given by
\begin{equation} \label{def_hatphi_jk_2}
\hat{\Gamma}_j = \mathbf{F}_j^{-1} \Gamma_j. 
\end{equation}
Since obviously
\begin{equation*}
\left\langle  \Gamma_j, \hat{\Gamma}_j \right\rangle = \mathbf{I}_j,
\end{equation*}
functions from $\hat{\Gamma}_j$ are duals to functions from $\Gamma_j$ in the space $V_{j+1}$. Since $\mathbf{F}_j^{-1}$ is not a sparse matrix,
these duals are not local.
We define a projection $P_j$ from $V_{j+1}$ onto $V_j$ by 
\begin{equation*}
P_j f = \sum_{k \in \mathcal{I}_j} \left\langle f, \hat{\phi}_{j,k} \right\rangle \phi_{j,k}.
\end{equation*} 

\begin{lem} \label{norm_composition}
There exist $p < 0.5$ such that a projection $P_j$ satisfies
\begin{equation} \label{norm_Pmn}
\left\| P_m P_{m+1} \ldots P_{n-1} \right\| \leq C \, 2^{p \, \left( n-m \right)},
\end{equation}
for all $2 \leq m < n $ and a constant $C$ independent on $m$ and $n$.
\end{lem}

\begin{proof}
Let $f \in V_{j+1}$, $a_k^j=\left\langle f, \hat{\phi}_{j,k} \right\rangle$, $\mathbf{a}_j= \left\{ a_k^j \right\}_{k \in \mathcal{I}_j}$, $j \geq 2$, and $\mathbf{S}_j: \mathbf{a}_{j+1} \mapsto \mathbf{a}_j$.
Then 
\begin{eqnarray*}
P_j f &=& \sum_{k \in \mathcal{I}_j} a_k^j \phi_{j,k} = \sum_{k \in \mathcal{I}_j} \left\langle f, \hat{\phi}_{j,k} \right\rangle \phi_{j,k} \\
\nonumber &=& \sum_{k \in \mathcal{I}_j} \sum_{l \in \mathcal{I}_{j+1} } a_l^{j+1} \left\langle \phi_{j+1,l}, \hat{\phi}_{j,k} \right\rangle \phi_{j,k}.
\end{eqnarray*}
Therefore
\begin{equation*}
a_k^j = \sum_{l \in \mathcal{I}_{j+1} } a_l^{j+1} \left\langle \phi_{j+1,l}, \hat{\phi}_{j,k} \right\rangle.
\end{equation*}
Let us denote
\begin{equation*}
S_{l,k}^j = \left\langle \hat{\phi}_{j,k}, \phi_{j+1,l} \right\rangle, \quad \mathbf{S}_j = \left\{ S_{l,k}^j \right\}_{l \in \mathcal{I}_{j+1}, k \in \mathcal{I}_j}
\end{equation*}
then we can write $\mathbf{a}_{j}=\mathbf{S}_j \mathbf{a}_{j+1}$ 
and due to Lemma \ref{decomposition_Phij} we have
\begin{equation*}
\mathbf{S}_j = \left\langle \hat{\Phi}_j, \Phi_{j+1} \right\rangle =  \left\langle \hat{\Phi}_j, \tilde{\mathbf{M}}_{j,0} \Phi_{j} + \tilde{\mathbf{M}}_{j,1} \Psi_{j} \right\rangle= \tilde{\mathbf{M}}_{j,0}.
\end{equation*}


Now, let us consider $f_n \in V_n$ and $f_m= P_m P_{m+1} \ldots P_{n-1} f_n$. Then $f_j$ can be represented by $f_j=\sum_{k \in \mathcal{I}_j} a_k^j \phi_j$ for $j=m, n$ and we set $\mathbf{a}_j= \left\{ a_k^j \right\}_{k \in \mathcal{I}_j}$. Since $\Phi_j$ is a Riesz basis of $V_j$, see \cite{Cerna2011}, there exist constants
$C_1$ and $C_2$ independent of $j$ such that
\begin{equation*}
C_1 \left\| \mathbf{a}_j \right\|_2 \leq \left\| \sum_{k \in \mathcal{I}_j} a_k^j \phi_{j,k} \right\| \leq C_2 \left\| \mathbf{a}_j \right\|_2.
\end{equation*}
Due to Lemma~\ref{norm_nasobeni_Nj} we have
\begin{eqnarray*}
\left\| f_m \right\| & \leq & C_2 \left\| \mathbf{a}_m \right\|_2 \leq C_2 \left\| \mathbf{S}_m \, \mathbf{S}_{m+1}  \ldots \mathbf{S}_{n-1} \right\|_2 \left\| \mathbf{a}_{n} \right\|_2 \\
\nonumber &=& C_2 \left\| \tilde{\mathbf{M}}_{m,0}^T \, \tilde{\mathbf{M}}_{m+1,0}^T  \ldots \tilde{\mathbf{M}}_{n-1,0}^T \right\|_2 \left\| \mathbf{a}_{n} \right\|_2 \\
\nonumber & \leq & C_2 \, 2^{p \left( n - m \right)} \, \left\| \mathbf{a}_n \right\|_2
\leq  C_1^{-1} \, C_2 \, 2^{p\left(n-m\right)} \, \left\| f_n \right\|. 
\end{eqnarray*}
Thus (\ref{norm_Pmn}) is proved.
$\qquad$ \end{proof}

\begin{thm} 
The sets $\Psi_j$ are Riesz bases of the spaces $W_j={\rm span} \ \Psi_j$, $j \geq 2$, with the condition numbers bounded independetly on $j$.
\end{thm}

\begin{proof}
The matrix $\mathbf{U}_j= \left\langle \Psi_j, \Psi_j \right\rangle$ is tridiagonal with entries
\begin{eqnarray*}
\left(\mathbf{U}_j\right)_{1,1} &=& \left(\mathbf{U}_j\right)_{2^j,2^j}=\frac{27}{320}, \\
\nonumber \left(\mathbf{U}_j\right)_{2,1} &=& \left(\mathbf{U}_j\right)_{1,2} = \left(\mathbf{U}_j\right)_{2^j-1,2^j}
= \left(\mathbf{U}_j\right)_{2^j,2^j-1} = \frac{47}{1920}, \\
\nonumber \left(\mathbf{U}_j\right)_{k,k} &=& \frac{1}{12}, \quad k=2, \ldots,  2^j-1, \\
\nonumber \left(\mathbf{U}_j\right)_{k,k+1} &=& \left(\mathbf{U}_j\right)_{k+1,k} = -\frac{1}{40}, \quad k=2, \ldots,  2^j-2, \\
\nonumber \left(\mathbf{U}_j\right)_{k,l} &=& 0, \quad {\rm otherwise}.
\end{eqnarray*}
Thus, $\mathbf{U}_j$ is strictly diagonally dominant and the assertion of the Theorem follows from Remark~\ref{remark1} and Gershgorin circle theorem.
$\qquad$ \end{proof}

\begin{thm} \label{theorem_riesz_1D}
The set 
\begin{equation*} 
\left\{2^{-2 } \phi_{2,k}, k \in \mathcal{I}_2 \right\} \cup \left\{2^{-j} \psi_{j,k}, j \geq 2, k \in \mathcal{I}_j \right\}
\end{equation*}
is a Riesz basis of $H_0^{1} \left( 0, 1 \right)$. 
\end{thm}

\begin{proof}
Using the same argument as in \cite{Jia2011} we conclude that (\ref{condition_Hq_1}) and (\ref{condition_Hq_2}) follows from the polynomial exactness of the scaling basis and the smoothness of basis functions and are satisfied for $H=L^2 \left( 0, 1\right)$ and
$H_q = H_0^q \left( 0, 1\right)$, $0<q<1.5$. Due to Lemma~\ref{norm_composition} the condition (\ref{norm_Pmn1}) is fulfilled. Therefore by Theorem~\ref{theorem_jia} the assertion of Theorem~\ref{theorem_riesz_1D} is proved. 
$\qquad$ \end{proof}

\begin{thm} \label{theorem_riesz_1Da}
The set 
\begin{equation*} 
\left\{\phi_{2,k}/ \left|\phi_{2,k}\right|_{H_0^1\left(0,1\right)}, k \in \mathcal{I}_2 \right\} \cup \left\{\psi_{j,k}/ \left|\psi_{j,k}\right|_{H_0^1\left(0,1\right)}, j \geq 2, k \in \mathcal{I}_j \right\},
\end{equation*}
where $\left| \cdot \right|_{H_0^1 \left( 0,1 \right)}$ denotes $H_0^1 \left(0,1\right)$--seminorm,
is a Riesz basis of $H_0^{1} \left( 0, 1 \right)$. 
\end{thm}

{\em Proof}.
We follow the proof of Lemma~2 in \cite{Cerna2014b}. From (\ref{dilatacepsi}) there exist constants $C_1$ and $C_2$ such that
\begin{equation} \label{ekvivalence_norem}
C_1 2^{j} \leq \left| \psi_{j,k} \right|_{H_0^1 \left(\Omega \right)} \leq C_2 2^{j}, \quad {\rm for} \, j \geq 2, \quad k \in \mathcal{I}_j,
\end{equation}
and
\begin{equation} \label{ekvivalence_norem_2}
C_1 2^{2} \leq \left| \phi_{2,k} \right|_{H_0^1 \left(\Omega \right)} \leq C_2 2^{2}, \quad {\rm for} \, k \in \mathcal{I}_2.
\end{equation}
Theorem~\ref{theorem_riesz_1D} implies that there exist constants $C_3$ and $C_4$ such that
\begin{equation} \label{Riesz10}
C_3 \left\| \mathbf{b} \right\|_2 \leq \left\| \sum_{k \in \mathcal{I}_2} a_{2,k} 2^{-2} \phi_{2,k} + \sum_{k \in \mathcal{I}_j, j \geq 2} b_{j,k} 2^{-j} \psi_{j,k} \right\|_{H_0^1 \left(0,1\right)} \leq C_4 \left\| \mathbf{b} \right\|_2,
\end{equation}
for any $\mathbf{b}=\left\{a_{2,k}, \ k \in \mathcal{I}_2 \right\} \cup \left\{b_{j,k}, j \geq 2, \ k \in \mathcal{J}_j \right\}$.
Using (\ref{ekvivalence_norem}), (\ref{ekvivalence_norem_2}), and (\ref{Riesz10}) we obtain 
\begin{equation*}
\left\| \mathbf{b} \right\|_2 \leq \frac{C_2}{C_3} \left\| \sum_{k \in \mathcal{I}_2} a_{2,k} \frac{\phi_{2,k}}{\left| \phi_{2,k} \right|_{H_0^1 \left(\Omega \right)}} + \sum_{k \in \mathcal{J}_j, j \geq 2} b_{j,k} \frac{\psi_{j,k}}{\left| \psi_{j,k} \right|_{H_0^1 \left(\Omega \right)}   } \right\|_{H_0^1 \left(0,1\right)}
\end{equation*}
and 
\begin{equation*} \left\| \mathbf{b} \right\|_2 \geq \frac{C_1}{C_4} \left\| \sum_{k \in \mathcal{I}_2} a_{2,k} \frac{\phi_{2,k}}{\left| \phi_{2,k} \right|_{H_0^1 \left(\Omega \right)}} + \sum_{k \in \mathcal{J}_j, j \geq 2} b_{j,k} \frac{\psi_{j,k}}{\left| \psi_{j,k} \right|_{H_0^1 \left(\Omega \right)}   } \right\|_{H_0^1 \left(0,1\right)}.
\qquad\endproof
\end{equation*}

\begin{thm}
The set $\Psi^{2D}$ normalized with respect to the $H^1$--seminorm is a Riesz basis of $H_0^{1} \left( \left(0, 1\right)^2 \right)$. 
\end{thm}

\begin{proof}
Recall that $\hat{\phi}_{j, k}$ are defined by (\ref{def_hatphi_jk}) and (\ref{def_hatphi_jk_2}). For $\mathbf{k}=\left(k_1, k_2 \right)$ let us define $\hat{\phi}_{j, \mathbf{k}}^2=\hat{\phi}_{j, k_1} \otimes \hat{\phi}_{j, k_2}$.
Then for $\mathbf{k}=\left(k_1, k_2 \right)$ and $\mathbf{l}=\left(l_1, l_2 \right)$ we have
\begin{equation*}
\left\langle \phi_{j, \mathbf{k}}^2, \hat{\phi}_{j, \mathbf{l}}^2 \right\rangle = \delta_{k_1, l_1} \delta_{k_2, l_2}, 
\end{equation*}
and $P_j^{2D}$ defined by
\begin{equation*}
P_j^{2D} f = \sum_{\mathbf{k} \in \mathcal{I}_j \times \mathcal{I}_j} \left\langle f , \hat{\phi}^2_{j, \mathbf{k}} \right\rangle \phi_{j, \mathbf{k}}^2
\end{equation*}
is a projection from $V_{j+1}^2$ onto $V_j^2$, where $V_j^2=V_j \otimes V_j$ for $j \geq 2$.
We denote $\mathbf{S}_j^{2D}= \tilde{\mathbf{M}}_{j,0}^T \otimes \tilde{\mathbf{M}}_{j,0}^T$. It is well-known that
for any matrix $\mathbf{B}$ we have
$\left\| \mathbf{B} \otimes \mathbf{B} \right\|_2 = \left\| \mathbf{B} \right\|_2^2$. 
Using this relation and the same arguments as in the proof of Lemma~\ref{norm_composition} we obtain
for $f_n \in V_n^2$ and $f_m= P_m^{2D} P_{m+1}^{2D} \ldots P_{n-1}^{2D} f_n$ the estimate:
\begin{eqnarray*}
\left\| f_m \right\| & \leq & C_1 \left\| \mathbf{a}_m \right\|_2 \leq C_2 \left\| \mathbf{S}_m^{2D} \, \mathbf{S}_{m+1}^{2D}  \ldots \mathbf{S}_{n-1}^{2D} \right\|_2 \left\| \mathbf{a}_{n} \right\|_2 \\
\nonumber &=& C_2 \left\| \left( \tilde{\mathbf{M}}_{m,0}^T \ldots \tilde{\mathbf{M}}_{n-1,0}^T \right) \otimes
\left( \tilde{\mathbf{M}}_{m,0}^T \ldots \tilde{\mathbf{M}}_{n-1,0}^T \right) \right\|_2 \left\| \mathbf{a}_{n} \right\|_2 \\
\nonumber & \leq & C_3 \, 2^{2p \left( n - m \right)} \, \left\| \mathbf{a}_n \right\|_2 \leq  C_4 \, 2^{2p\left(n-m\right)} \, \left\| f_n \right\|
\end{eqnarray*}
with $2p<1$. Hence by Theorem~\ref{theorem_jia} the assertion of the theorem is proved. 
$\qquad$ \end{proof}

\section{Quantitative properties of constructed bases}
\label{quantitative_properties}

In this section, we present the condition numbers of the stiffness matrices for the Helmholtz equation 
\begin{equation} \label{Helmholtz_eq}
- \epsilon \Delta u + a u  = f \, \, \text{ on} \, \, \, \Omega_d, \quad u=0 \, \, \text{on} \, \, \partial \Omega_d,
\end{equation}
where $\Delta$ is the Laplace operator, $\epsilon$ and $a$ are positive constants.

The variational formulation is
\begin{equation} \label{discrete1}
\mathbf{A} \mathbf{u} =\mathbf{f}, 
\end{equation}
where 
\begin{equation*}
\mathbf{A} = \epsilon \left\langle \nabla \Psi, \nabla \Psi \right\rangle + a \left\langle \Psi, \Psi \right\rangle, \quad
\nonumber u = \left( \mathbf{u} \right)^T \Psi, \quad \mathbf{f} = \left\langle f, \Psi \right\rangle. 
\end{equation*}
An advantage of discretization of elliptic equation (\ref{Helmholtz_eq}) using a wavelet basis is that the system (\ref{discrete1}) can be simply preconditioned by a diagonal preconditioner \cite{Dahmen1992}.  Let $\mathbf{D}$ be a matrix of diagonal elements of the matrix $\mathbf{A}$, i.e. $\mathbf{D}_{\lambda, \mu}=
\mathbf{A}_{\lambda, \mu} \delta_{\lambda, \mu}$, where $\delta_{\lambda, \mu}$ denotes Kronecker delta.
Setting
\begin{equation*}
\tilde{\mathbf{A}} = \left(\mathbf{D}\right)^{-1/2} \mathbf{A} \left(\mathbf{D}\right)^{-1/2}, \quad
\nonumber \tilde{\mathbf{u}} = \left(\mathbf{D}\right)^{1/2} \mathbf{u}, \quad \tilde{\mathbf{f}}= \left(\mathbf{D}\right)^{-1/2}\mathbf{f},
\end{equation*}
we obtain the preconditioned system $\tilde{\mathbf{A}} \tilde{\mathbf{u}}=\tilde{\mathbf{f}}$.
It is known \cite{Dahmen1992} that there exist a constant $C$ such that
$\textrm{ cond} \, \tilde{\mathbf{A}} \leq C < \infty$.

Let $\Psi_s$ be defined by (\ref{definition_Psi}) for $d=1$ and similarly for $d>1$. We define
\begin{equation*}
\mathbf{A}_s = \epsilon \left\langle \nabla \Psi_s, \nabla \Psi_s \right\rangle + a \left\langle \Psi_s, \Psi_s \right\rangle, \quad
\nonumber u_s = \left( \mathbf{u}_s \right)^T \Psi_s, \quad \mathbf{f}_s = \left\langle f, \Psi_s \right\rangle.
\end{equation*}
Let $\mathbf{D}_s$ be a matrix of diagonal elements of the matrix $\mathbf{A}_s$, i.e. $\left(\mathbf{D}_s\right)_{\lambda, \mu}=
\left(\mathbf{A}_s\right)_{\lambda, \mu} \delta_{\lambda, \mu}.$
We set
\begin{equation} \label{preconditioned_matrix}
\tilde{\mathbf{A}}_s = \left(\mathbf{D}_s\right)^{-1/2} \mathbf{A}_s \left(\mathbf{D}_s\right)^{-1/2}, \quad
\nonumber \tilde{\mathbf{u}}_s = \left(\mathbf{D}_s\right)^{1/2} \mathbf{u}_s, \quad \tilde{\mathbf{f}}_s = \left(\mathbf{D}_s\right)^{-1/2}\mathbf{f}_s
\end{equation}
and we obtain preconditioned finite-dimensional system
\begin{equation} \label{discrete_finite_problem}
\tilde{\mathbf{A}}_s \tilde{\mathbf{u}}_s=\tilde{\mathbf{f}}_s.
\end{equation}
Since $\tilde{\mathbf{A}}_s$ 
is a part of the matrix $\tilde{\mathbf{A}}$ that is symmetric and positive definite, we have also
\begin{equation*}
\textrm{ cond} \, \tilde{\mathbf{A}}_s \leq C.
\end{equation*}
The condition numbers of the stiffness matrices $\mathbf{A}_s$ for $d=1$ and $d=2$ are shown in Table~\ref{tab1}. 
Although it was not proved in this paper that using appropriate tensorising of 1D wavelet basis we obtain wavelet basis
in 3D, we listed the condition numbers of the stiffness matrices $\mathbf{A}_s$ for 3D case in Table~\ref{tab2}. 
The condition numbers for several constructions of quadratic spline wavelet bases and various values of parameters $\epsilon$ and $a$ are compared in Table~\ref{tab_srovnani}. $CF_2$ denotes the construction from this paper with the coarsest level $2$, $CF_3$ denotes the construction from this paper with the coarsest level $3$. $CF_2^{ort}$ and $CF_3^{ort}$ are bases from this paper with the orthogonalization of the scaling functions on the coarsest level. $P_2$ and $P_3$ refers to quadratic spline wavelet bases adapted to homogeneous Dirichlet boundary condition from \cite{Primbs2010} and $D_2$ and $D_3$ refers to bases from \cite{Dijkema2009}.

\begin{table} [!ht] 
\caption{The condition numbers of the stiffness matrices $\mathbf{A}_s$ of the size $N \times N$ corresponding to multiscale wavelet bases with $s$ levels of wavelets for the one-dimensional and the two-dimensional Poisson equation. }
\label{tab1} 
\begin{tabular}{ r | r  r  r   r |   r  r  r  r }
 & \multicolumn{4}{c|}{1D } & \multicolumn{4}{|c}{2D } \\
$s$ & $N$ & $\lambda_{min}$ & $\lambda_{max}$ & ${\rm cond} \mathbf{A}_s$ & $N$ & $\lambda_{min}$ & $\lambda_{max}$ & ${\rm cond} \mathbf{A}_s$ \\
\hline 
1 & 8   & 1.38 & 0.50 & 2.77 &        64 & 0.25 & 1.88  &   7.5  \\
2 & 16  & 1.41 & 0.50 & 2.83 &       256 & 0.19 & 2.08  &  11.1 \\
3 & 32  & 1.42 & 0.50 & 2.83 &     1 024 & 0.16 & 2.17  &  13.7  \\
4 & 64  & 1.42 & 0.50 & 2.84 &     4 096 & 0.14 & 2.20  &  15.4  \\
5 & 128 & 1.42 & 0.50 & 2.84 &    16 384 & 0.13 & 2.22  &  16.6  \\
6 & 256 & 1.42 & 0.50 & 2.84 &    65 536 & 0.13 & 2.23  &  17.4  \\
7 & 512 & 1.42 & 0.50 & 2.84 &   262 144 & 0.12 & 2.23  &  17.9   \\
8 & 1024& 1.42 & 0.50 & 2.84 & 1 048 576 & 0.12 & 2.23  &  18.3  \\
\end{tabular}
\end{table}

\begin{table} [!ht]  
\caption{The condition numbers of the stiffness matrices $\mathbf{A}_s$ of the size $N \times N$ corresponding to multiscale wavelet bases with $s$ levels of wavelets for the three-dimensional Poisson equation. }
\label{tab2} 
\begin{tabular}{ r|  r  r  r   r }
$s$ & $N$ & $\lambda_{min}$ & $\lambda_{max}$ & ${\rm cond} \mathbf{A}_s$ \\
\hline 
1 &     512 & 0.15 & 3.23 & 47.4 \\
2 &    4096 & 0.04 & 3.69 & 85.0 \\
3 &   32768 & 0.03 & 3.83 & 113.8 \\
4 &  262144 & 0.03 & 3.87 & 132.9 \\
5 & 2097152 & 0.03 & 3.89 & 145.3 \\
\end{tabular}
\end{table}

\begin{table} [!ht]  
\caption{The condition numbers of the stiffness matrices $\mathbf{A}_s$ of the size $65 536 \times 65 536$ for several choices of $\epsilon$ and $a$ for our bases and bases from \cite{Dijkema2009, Primbs2010}.}
\label{tab_srovnani}   
\begin{tabular}{ r  r|  r  r  r   r   r  r  r  r }
$\epsilon$ & $a$ & $CF_2$ & $CF_3$ & $CF_2^{ort}$ & $CF_3^{ort}$ & $P_2$ & $P_3$ & $D_2$ & $D_3$\\
\hline 
1000 & 1 & 17.4 &16.3 & 17.1 & 16.4 & 116.3 & 98.5 & 116.3 & 98.4\\
1 & 0 & 17.4 & 16.7 & 17.1 & 16.4 &   116.3 & 98.5& 116.3 & 98.4 \\
1 & 1 & 17.4 & 16.7 & 17.1 & 16.4 &  116.6 & 98.5 & 116.6 & 98.5 \\
$10^{-3}$ & 1 & 72.1 & 35.9 & 35.6 & 22.5  & 328.1 & 139.2 & 328.1 & 139.2 \\
$10^{-6}$ & 1 & 746.0 & 577.0 & 425.7 & 287.6 & 1878.0 & 1115.4 & 1878.0 & 1115.4\\
$0$       & $1$ & 872.6 & 687.4 & 511.0 & 351.5  & 2034.5 & 1251.3 & 2034.6 & 1251.4 \\
\end{tabular}
\end{table}

\section{Numerical example}

The constructed wavelet basis can be used for solving various types of problems. Let us mention for example solving partial differential and integral equations by adaptive wavelet method \cite{Cohen2001, Cohen2002}. In this section we use constructed wavelet basis in the wavelet-Galerkin method and an adaptive wavelet method. 

\subsection{Multilevel Galerkin method}
We consider the problem (\ref{Helmholtz_eq}) with $\Omega_2$, $\epsilon=1$ and $a=0$. The right-hand side $f$ is such that the solution $u$ is given by:
\begin{equation} \label{example_solution}
u \left(x,y\right)=v\left( x\right) v \left( y \right), \quad
v \left( x\right) = x \left( 1 - e^{50x-50} \right).
\end{equation}
We discretize the equation using the Galerkin method with wavelet basis constructed in this paper and we obtain
discrete problem $\tilde{\mathbf{A}}_s \tilde{\mathbf{u}}_s = \tilde{\mathbf{f}}_s$. We solve it by conjugate gradient method using a simple multilevel approach similarly as in \cite{Cerna2015, Jia2011}: 

\medskip
1. Compute $\tilde{\mathbf{A}}_s$ and $\tilde{\mathbf{f}}_s$, choose $\mathbf{v}_0$ of the length $4^{2}$.

\medskip
2. For $j=0, \ldots, s$ find the solution $\tilde{\mathbf{u}}_j$ of the system $\tilde{\mathbf{A}}_j \tilde{\mathbf{u}}_j = \tilde{\mathbf{f}}_j$ by conjugate gradient method with initial vector $\mathbf{v}_j$ defined for $j \geq 1$ by
      \begin{equation*}
      \left(\mathbf{v}_j\right)=\left\{ \begin{array}{cl}
      \tilde{\mathbf{u}}_{j-1}, & i=1, \ldots, k_{j}, \\
      0, & i= k_{j}, \ldots, k_{j+1}, \\
      \end{array} \right.
      \end{equation*}
      where $k_{j}=2^{2 \left( j +1\right)}$.
\bigskip

Let $u$ be the exact solution of (\ref{Helmholtz_eq}) and 
\begin{equation*}
u^*_s= \left( \tilde{\mathbf{u}}^*_s\right)^T \left( \mathbf{D}_s\right)^{-1/2} \Psi_s,
\end{equation*}
where $\tilde{\mathbf{u}}^*_s$ is the exact solution of the discrete problem (\ref{discrete_finite_problem}). It is known \cite{Urban2009} that 
\begin{equation} \label{estimate_convergence_rate}
\left\| u - u^*_s \right\|_{H^m\left( \Omega \right)} \leq C 2^{-\left(3-m\right)s}.
\end{equation}
Let $u_s$ be an approximate solution obtained by multilevel Galerkin method with $s$ levels of wavelets.
It was shown in \cite{Cerna2015a, Cerna2015b} that 
if we use the criterion for terminating iterations $\left\| \mathbf{r}_s \right\|_2 \leq C 2^{-2s}$, where $\mathbf{r}_s:= \tilde{\mathbf{A}}_s \tilde{\mathbf{u}}_s - \tilde{\mathbf{f}}_s$, then we achieve for $u_s$ the same convergence
rate as for $u^*_s$. In our example, for the given number of levels $s$ we use the criterion $\left\| \mathbf{r}_j \right\|_2 \leq 10^{-4} 2^{-2s}$, $j=0, \ldots, s,$ for terminating iterations in each level. 

We denote the number of iterations on the level $j$ as $M_j$. It is known \cite{Urban2009} that employing the discrete wavelet transform one CG iteration can be performed with complexity of the order $\mathcal{O} \left(N\right)$, where $N \times N$ is the size of the matrix. Therefore the number of operations needed to compute one CG iteration on the level $j$ requires about one quarter of operations needed to compute one CG iteration on the level $j+1$, we compute the total number of equivalent iterations by
\begin{equation*}
M = \sum_{j=0}^s \frac{M_j}{4^{s-j}}.
\end{equation*}
The results are listed in Table~\ref{tab6}. It can be seen that the number of conjugate gradient iterations is quite small and that 
\begin{equation*}
\frac{\left\|u_s - u\right\|_{\infty}}{\left\|u_{s+1} - u\right\|_{\infty}} \approx
\frac{\left\|u_s - u\right\|}{\left\|u_{s+1} - u\right\|} \approx \frac{1}{8},
\end{equation*}
i.e. that the order of convergence is $3$. It corresponds to (\ref{estimate_convergence_rate}). In Table~\ref{tab6}
$CF$ denotes the construction in this paper, $D, P$ denote constructions from \cite{Dijkema2009, Primbs2010}. The constructions from \cite{Dijkema2009} and \cite{Primbs2010} differ only in the definition of boundary wavelets and the results were the same.

\begin{table} [!ht] 
\caption{Number of iterations and error estimates for multilevel conjugate gradient method.}
\label{tab6}  
\begin{tabular}{r  r |r r r |r r r }
\hline
 & & & $CF$ & & & $P, D$ &  \\
 $s$  & $N$ & $M$ & $\left\|u_s - u\right\|_{\infty}$ & $\left\| u_s - u \right\| $ & $M$ & $\left\|u_s - u\right\|_{\infty}$ & $\left\| u_s - u \right\| $ \\
\hline 
0 &       16 & 10.00 & 5.42e-1 & 1.01e-1 & 10.00 & 5.42e-1 & 1.01e-1   \\
1 &       64 & 18.50 & 3.19e-1 & 4.54e-2 & 27.50 & 3.19e-1 & 4.54e-2   \\
2 &       256 & 21.63 & 1.32e-1 & 1.26e-3& 48.88 & 1.32e-1 & 1.26e-3   \\
3 &     1 024 & 23.66 & 2.60e-2 & 2.02e-3& 59.22 & 2.60e-2 & 2.02e-3   \\
4 &     4 096 & 23.00 & 2.91e-3 & 2.45e-4& 59.38 & 2.91e-3 & 2.45e-4   \\
5 &    16 384 & 20.89 & 4.06e-4 & 2.89e-5& 50.76 & 4.06e-4 & 2.89e-5   \\
6 &    65 536 & 18.37 & 5.35e-5 & 3.41e-6& 39.44 & 5.35e-5 & 3.41e-6   \\
7 &   262 144 & 15.68 & 6.82e-6 & 4.23e-7& 29.92 & 6.84e-6 & 4.23e-7   \\
8 & 1 048 576 & 13.02 & 8.63e-7 & 5.28e-8& 21.50 & 8.64e-7 & 5.29e-8   \\
9 & 4 194 304 & 10.35 & 1.08e-7 & 6.59e-9& 17.66 & 1.09e-7 & 6.73e-9   \\
\hline
\end{tabular}
\end{table}

\subsection{Adaptive wavelet method}

We compare the quantitative behavior of the adaptive wavelet method with our bases and bases from \cite{Dijkema2009, Primbs2010}. We consider the equation (\ref{Helmholtz_eq}) with $\epsilon=1$ and $a=0$ for $\Omega_2$ with the solution $u$ given by (\ref{example_solution}).
Then the solution exhibits a sharp gradient near the point $[1,1]$. 
We solve the problem by the adaptive wavelet method proposed in \cite{Cohen2001, Cohen2002} with the matrix-vector multiplication from \cite{Cerna2013}. The coarsest level of the wavelet basis is $j_0=2$ and we use wavelets up to the scale $\left| \lambda \right| \leq 10$. The convergence history is shown in Figure~\ref{comparison_adaptive}.
It can be seen that the convergence rate is similar for all bases. However, the number of iterations needed to resolve the problem with desired accuracy is significantly smaller for the new wavelet basis. Moreover, due to the shorter support of the wavelets, the stiffness matrix is sparser and thus one iteration requires smaller number of operations.  
The number of iterations is much larger in comparison with the results obtained by the multilevel Galerkin method in Table \ref{tab6}, but the number of basis functions is significantly smaller. 

\begin{figure} [t] 
  \includegraphics[height=.2\textheight]{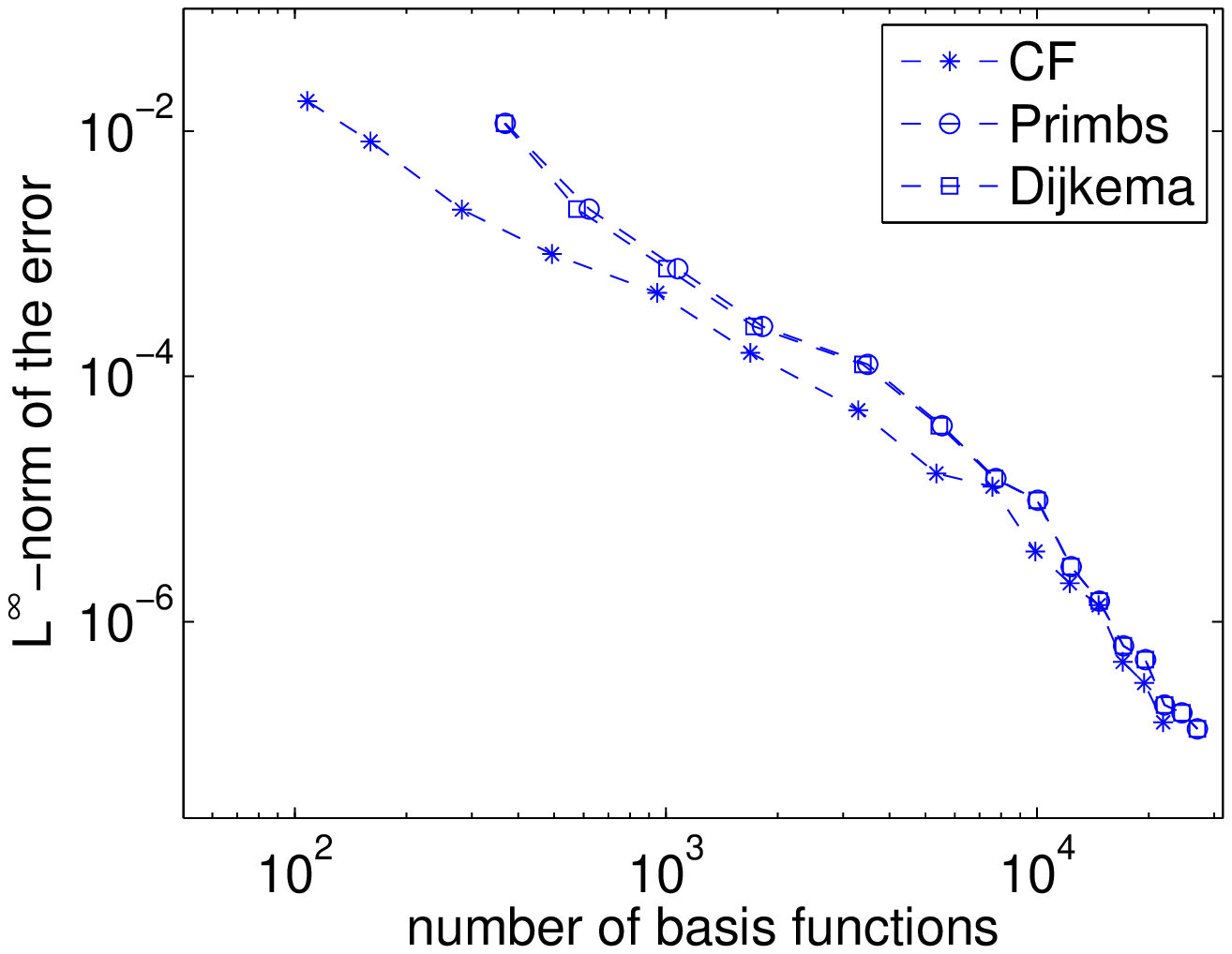} \qquad
  \includegraphics[height=.2\textheight]{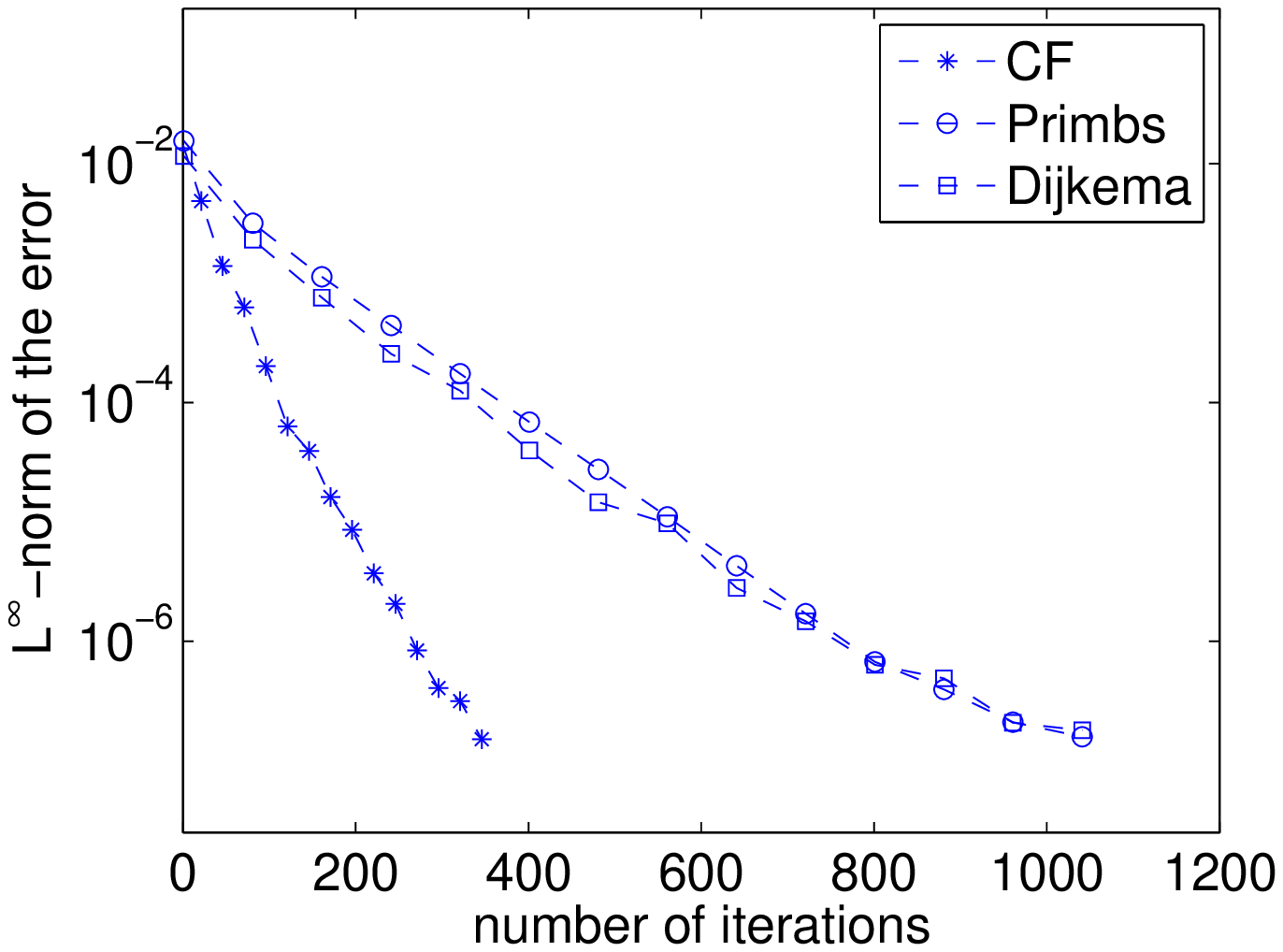}
  \caption{The convergence history for adaptive wavelet scheme with various wavelet bases.}
  \label{comparison_adaptive}
\end{figure}

\end{document}